\documentclass[11pt]{article}
\usepackage{amsmath,amsthm,amsfonts,amssymb,mathrsfs,bm}
\usepackage{amsmath,amsthm,amsfonts,amssymb,bm,wasysym}
\usepackage{epsfig}
\usepackage[usenames]{color}
\usepackage{verbatim}
\usepackage{hyperref}
\usepackage{multicol}
\usepackage{comment}
\usepackage{float}
\usepackage{graphicx}
\usepackage{centernot}
\usepackage{color}
\usepackage[]{algorithm2e}

\graphicspath{{./Pics}}

\usepackage[colorinlistoftodos]{todonotes}
\usepackage[normalem]{ulem}

\parskip \medskipamount
\newtheorem{theorem}{Theorem}[section]
\newtheorem{definition}[theorem]{Definition}

\numberwithin{equation}{section}
\newtheorem{lemma}[theorem]{Lemma}
\newtheorem{proposition}[theorem]{Proposition}
\newtheorem{corollary}[theorem]{Corollary}

\newtheorem{remark}[theorem]{Remark}

\newcommand{\Var}{{\mathrm{Var}}}

\newtheorem{question}[theorem]{Question}

\numberwithin{equation}{section}

\def\N{\mathbb{N}}
\def\Z{\mathbb{Z}}

\def\R{\mathbb{R}}

\def\EE{\mathcal{E}}

\newcommand{ \cL}{ \mathcal L }

\def\PP{\mathcal{P}}

\renewcommand{\phi}{\varphi}
\renewcommand{\epsilon}{\varepsilon}

\allowdisplaybreaks

\newcommand{ \mix}{ t_{\mathrm{mix}} }

\newcommand{\1}{{\text{\Large $\mathfrak 1$}}}

\newcommand{ \rel}{ t_{\mathrm{rel}} }

\renewcommand{\emptyset}{\varnothing}

\newcommand{\til}{\widetilde}

\newcommand{\tstop}{t_{\mathrm{stop}}}

\newcommand{\pr}[1]{\mathbb{P}\!\left(#1\right)}
\newcommand{\estart}[2]{\mathbb{E}_{#2}\!\left[#1\right]}
\newcommand{\prstart}[2]{\mathbb{P}_{#2}\!\left(#1\right)}
\newcommand{\prcond}[3]{\mathbb{P}_{#3}\!\left(#1\;\middle\vert\;#2\right)}
\newcommand{\econd}[2]{\mathbb{E}\!\left[#1\;\middle\vert\;#2\right]}
\newcommand{\escond}[3]{\mathbb{E}_{#3}\!\left[#1\;\middle\vert\;#2\right]}

\newcommand{\E}{{\mathbb{E}}}

\def\P{\mathbb{P}}

\newcommand{\norm}[1]{\left\| #1 \right\|}

\newcommand{\tn}{|\kern-.1em|\kern-0.1em|}

\newcommand\be{\begin{equation}}
\newcommand\ee{\end{equation}}

\def\eps{\varepsilon}

\newcommand{\tv}[1]{\left\|#1\right\|_{\rm{TV}}}

\newcommand{\SRW}{{\mathrm{SRW}}}
\newcommand{\fu}{{\mathrm{full}}}

\newcommand{\mixsrw}{ t_{\mathrm{mix}}^{\mathrm{SRW}} }

\newcommand{\mixi}{ t_{\mathrm{mix}}^{(\infty)} }
\newcommand{\mixisrw}{ t_{\mathrm{mix}}^{\mathrm{SRW},(\infty)} }

\newcommand{\mixf}{ t_{\mathrm{mix}}^{\mathrm{full}} }
\newcommand{\mixfmup}{ t_{\mathrm{mix}}^{\mathrm{full},(\mu,p)} }
\newcommand{\mixfi}{ t_{\mathrm{mix}}^{\mathrm{full},(\infty)} }
\newcommand{\mixfmupi}{ t_{\mathrm{mix}}^{\mathrm{full},(\mu,p),(\infty)} }

\newcommand{\hit}{t_{\mathrm{hit}}}
\newcommand{\hitsrw}{t_{\mathrm{hit}}^{\mathrm{SRW}}}
\newcommand{\hitf}{t_{\mathrm{hit}}^{\mathrm{full}}}
\newcommand{\hitfmup}{t_{\mathrm{hit}}^{\mathrm{full},(\mu,p)}}

\newcommand{\sfrac}[2]{\mbox{\small $\frac{#1}{#2}$}}
\newcommand{\ssfrac}[2]{\mbox{\footnotesize $\frac{#1}{#2}$}}
\newcommand{\half}{\ssfrac{1}{2}}

\begin{document}

\title{\bf A comparison principle for random walk on dynamical percolation}

\author{Jonathan Hermon
\thanks{
University of British Columbia, Vancouver, Canada.  E-mail: {jhermon@math.ubc.ca}.}
\and Perla Sousi
\thanks{
University of Cambridge, Cambridge, UK.  E-mail: {p.sousi@statslab.cam.ac.uk}. This work was supported by the Engineering and Physical Sciences Research Council: JH by EP/L018896/1 and PS by EP/R022615/1.}
}
\date{}

\maketitle

\begin{abstract}
We consider the model of random walk on dynamical percolation introduced by Peres, Stauffer and Steif in~\cite{PSSsub}. We obtain comparison results for this model for hitting and mixing times and for the spectral-gap
 and log-Sobolev constant with the corresponding quantities for simple random walk on the underlying graph $G$, for  general graphs. When  $G$ is the torus $\Z_n^d$, we recover the results of Peres et al.\ and we also extend them to the critical case. We also obtain bounds in the cases where $G$ is a transitive graph of moderate growth and also when it is the hypercube. 
\newline
MSC 2010 \emph{subject classifications.} Primary 60F05, 60G50.
\end{abstract}

\section{Introduction}

In this paper we consider the model of random walk on a dynamically evolving environment introduced in \cite{PSSsub}. Fix a base graph $G=(V,E)$ and let each edge $e$ refresh at rate $\mu$ to open with probability $p$ and closed with probability $1-p$ independently of other edges and previous states of the same edge. Let $X$ be a continuous time random walk that moves as follows: at rate $1$ it chooses one of its neighbours uniformly at random and only jumps there if the edge connecting the neighbour to its current location is open. Otherwise it stays in place. We denote the state of the full system at time $t$ by $(X_t,\eta_t)$, where $X_t\in V$ and $\eta_t\in \{0,1\}^E$ with $0$ representing a closed edge and~$1$ an open one. We refer to $\eta_t$ as the environment at time $t$. We emphasise  that $(X_t,\eta_t)_{t \ge 0}$ is Markovian, while the location of the walker~$(X_t)_{t \ge 0}$ is not.
One readily checks that $\pi_{\mathrm{full},p}=\pi\times \pi_p$ is the unique stationary distribution
and that the process is reversible; here $\pi$ is the degree biased distribution on $V$, i.e.\ $\pi(x)=\deg(x)/(2|E|)$ for all $x$ and
$\pi_p$ is product measure of $\mathrm{Ber}(p)$ on the edges. Moreover, even if the environment process
$\{\eta_t\}_{t \ge 0}$ is fixed,  $\pi$ is a stationary distribution for the resulting time inhomogeneous Markov process $(X_t)_{t \ge 0}$.

There has been a lot of interest recently in studying these processes. The case where $G=\Z_n^d$ was studied in~\cite{PSSsub, PSSsuper}.  The subcritical regime seems to be fully understood \cite{PSSsub}, while the supercritical case still poses challenges. In~\cite{conf}, the authors established precise mixing time results for the non-backtracking random walk on a dynamic configuration model and in \cite{ER} the authors studied the case where $G$ is the complete graph. In this paper, we study mixing and hitting times for random walk on dynamical percolation on general graphs. We develop general machinery that allows for a comparison of these quantities with the corresponding ones for simple random walk (SRW) on the base graph $G$. Recall that the simple random walk (SRW) on $G=(V,E)$ is a Markov chain on $V$ with transition probabilities given by $P(x,y)=\1(\{x,y\} \in E)/\deg (x) $, where $\deg (x)$ is the degree of $x$ (i.e.\ at each step the walk picks a neighbour uniformly at random and jumps to it). We emphasise that from now on, whenever we write SRW, we refer to simple random walk on the (static) graph $G$, i.e.\ where all edges in $E$ are present. Below, we shall consider its continuous-time version with jump rate 1. We note that our upper bounds on the hitting and mixing times hold for all $\mu$ and all $p \in (0,1] $ with no difference between the subcritical and supercritical regimes. This is in sharp contrast to previous works.

Let $P$ be a transition matrix  with stationary distribution $\pi$. We define the total variation and $L_\infty$ mixing times as follows
\begin{align*}
        \mix(\epsilon) &= \min\{ t\geq 0: \max_{x} \tv{P^t(x,\cdot) - \pi} \leq \epsilon\} \\
        \mixi(\epsilon) &= \min\{ t\geq 0: \max_{x} \|P^t(x,\cdot) - \pi\|_{\infty, \pi} \leq \epsilon\},
\end{align*}
where the total variation and $L_{\infty}$ norms of a signed measure $a$ are given by $\|a  \|_{\rm{TV}}=\frac 12 \sum_x |a(x)| $ and  $\|a  \|_{\infty,c}= \max_x |a(x)/c(x)| $. 
We are primarily interested in the total variation and $L_{\infty}$ mixing times of the full system and we denote them by 
\[
\mixf(\epsilon)=\mixfmup(\epsilon) \quad \text{and}  \quad
\mixfi(\epsilon)=\mixfmupi(\epsilon).
\]
 We denote the corresponding mixing times for the SRW on $G$ by $\mixsrw(G,\eps) $ and $\mixisrw(G,\eps) $. We omit $G$ when clear from context, and $\eps$ when $\eps=1/4$.
 
 Let $\mathbb{P}_{x,\eta}$ be the law of the full process, started from initial environment $\eta$ and initial location $x$ for the walk. We denote the corresponding expectation by $\mathbb{E}_{x,\eta}$. When $p$ and $\mu$ are not clear from context, we write  $\mathbb{P}_{x,\eta}^{(\mu,p)}$ and   $\mathbb{E}_{x,\eta}^{(\mu,p)}$. We write $\mathbb{P}_{x,\eta}^t $ as a shorthand for the law of the full process at time $t$, $\prstart{(X_t,\eta_t) = (\cdot, \cdot)}{x,\eta} $.  

We are also interested in \emph{hitting times} by the full process of the form
\[\hitf=\hitfmup=\max_{x,y \in V ,\eta \in \{0,1\}^E}\mathbb{E}_{x,\eta}[T_{\{y\} \times  \{0,1\}^E}], \]   
 where for a set $A \subset V \times \{0,1\}^E $ its \emph{hitting time} $T_A=\inf\{t:(X_t,\eta_t) \in A \}$ is defined to be the first time the process visits the set $A$. Our Theorem \ref{thm:2} bounds $\hitf$ in terms of the hitting times for the SRW which are denoted by $\hitsrw(G)=\max_{x,y \in V } \mathbb{E}_{x}^{\SRW}[T_y] $.

For functions $f,g$ we will write $f(n) \lesssim g(n)$ if there exists a constant $c > 0$ such that $f(n) \leq c g(n)$ for all $n$.  We write $f(n) \gtrsim g(n)$ if $g(n) \lesssim f(n)$.  Finally, we write $f(n) \asymp g(n)$ if both $f(n) \lesssim g(n)$ and $f(n) \gtrsim g(n)$. {We write $\asymp_{a}, \lesssim_{a}$ and $\gtrsim_a$ when the implied constant depends on $a$.}

Peres and Steif \cite{SteifPC} asked whether $\hitfmup \lesssim_{\mu,p}|V|^3 $, which is the natural analog of the classical bound $\hitsrw\lesssim |V|^3$ (see e.g., \cite{aldous,levin}). In the following theorem we give an upper bound on $\hitfmup $ in terms of $\hitsrw$. Using  $\hitsrw\lesssim |V|^3$  this  answers affirmatively their question.

\begin{theorem}[Hitting time comparison with SRW]
\label{thm:2}
For every $\mu$ there exists a positive constant~$c_1$ such that for all 
graphs $G$ and all $p$ we have that
\begin{equation}
\label{e:hitmain1}
\hit^{\mathrm{full},(\mu,p)} \leq \frac{c_1}{p} \cdot\hitsrw.
\end{equation}
Moreover, there exists a constant $c_2$ so that for all graphs $G$ and all $(\mu,p) \in (0,1]^2 $ we have that 
\begin{equation}
\label{e:hitmain2}
\hitfmup \leq c_2( \mu^{-1} \hit^{\mathrm{full},(1,p)}+\mixfmup ). 
 \end{equation}
\end{theorem}

We believe the $\mixfmup$ term in~\eqref{e:hitmain2} in most cases satisfies $\mixfmup \lesssim  \mu^{-1} \hit^{\mathrm{full},(1,p)}  $ and can thus be removed from ~\eqref{e:hitmain2}, see Remark 2 in Section~\ref{sec:remarks}.

The Dirichlet form associated to the transition matrix $P$ (respectively, generator $\cL$) is defined to be
 \[
 \EE_P(f,f)  = \frac{1}{2} \sum_{x,y} \pi(x) P(x,y) (f(x) - f(y))^2=\pi\left( (I-P)f \cdot f\right)
 \]
for all $f:\Omega \to \R$ (respectively, $\EE_{\cL}(f,f)  = \frac{1}{2} \sum_{x,y} \pi(x) \cL(x,y) (f(x) - f(y))^2 = \pi\left(-\cL f\cdot f \right)$). 

For $\epsilon>0$ we denote the spectral profile
\begin{equation}
\label{e:la}
\Lambda(\eps)=\min \{\EE(h,h) : h \in \R^{\Omega}, \, \Var_{\pi}(h)=1,\, \pi( \mathrm{supp}(h)) \le \eps   \},
\end{equation}
 where $ \mathrm{supp}(h)=\{x \in \Omega :h(x) \neq 0 \}$ is the \emph{support} of $h$ and $\Var_{\pi}(h)=\E_{\pi}[(h-\E_{\pi}h)^{2}] $ is the variance w.r.t.\ $\pi$.

For $\epsilon>0$, the $\epsilon$-\emph{spectral-profile} time is given by 
\begin{equation}
\label{e:profiles}
t_{\mathrm{spectral-profile}}(\eps)= \int_{4\pi_*}^{4/\eps} \frac{2d \delta}{\delta \Lambda(\delta) },
\end{equation}
where $\pi_*=\min_x \pi(x)$.

Goel et al.\ \cite{spectral} showed that $\mixi(\eps) \le t_{\mathrm{spectral-profile}}(\eps) $ (this refines the evolving sets bound of Morris and Peres \cite{evolving}). Let $t_{\mathrm{spectral-profile}}^{\SRW}$ be $t_{\mathrm{spectral-profile}}$ defined w.r.t.\ simple random walk.  We recall a result of Kozma \cite{kozma}  that for all (simple) finite graphs  $t_{\mathrm{spectral-profile}}^{\SRW} \lesssim  \mixisrw \log \log |V| $. However, there are many families of graphs for which  $t_{\mathrm{spectral-profile}}^{\SRW} \asymp \mixisrw \asymp \mixsrw$. For instance, this is the case for the hypercube and for vertex-transitive graphs of moderate growth.

\begin{theorem}[Mixing time comparison with SRW]
\label{thm:1}
There exists a positive constant $c_1$ such that for all graphs $G$ and all $(\mu,p) \in (0,1]^2 $ we have for all $\epsilon \in (0,1)$
\begin{equation}
\label{e:mixmain}
\mixfmupi(\epsilon) \leq \frac{c_1}{\mu p} \cdot  t_{\mathrm{spectral-profile}}^{\SRW}(\epsilon) +\frac{c_1}{\mu}|\log\left(1-p\right)|.
\end{equation}
\end{theorem}

\begin{remark}\label{rem:1-p}
\rm{
        We note that the term $|\log \left(1-p \right)| /\mu$ is necessary, because the $L_\infty$ mixing time
        of the environment is at least of this order. However, if we were considering the total variation mixing time of the full system, then we would get rid of this extra term. We explain this at the end of the proof of this theorem in Section~\ref{s:proof1}.
        }
\end{remark}

We recall that the \emph{spectral-gap} of a reversible transition matrix $P$ (resp.\ generator $\cL$) is defined as the smallest positive eigenvalue of $I-P$ (resp.\ $-\cL$). Its inverse is called the \emph{relaxation-time}. We denote the relaxation-time of the full  process by $\rel^{\mathrm{full}} =\rel^{\mathrm{full},(\mu,p)} $ and that of the SRW by $\rel^{\SRW} $. We recall that generally for a continuous-time reversible Markov chain on a finite state space $\Omega$ with stationary distribution $\pi$ the relaxation-time determines the asymptotic exponential rate of convergence to equilibrium  in the following precise sense (cf., \cite[Lemmas 20.5 and 20.11]{levin}) for all $\eps \in (0,1)$ 
\[\rel | \log \eps | \le  \mix(\eps/2) \le   \mixi(\eps) \le \rel | \log (\pi_* \eps) |, \quad \text{where} \quad \pi_*:=\min_{x \in \Omega }\pi(x). \] 

The inequality above together with Theorem~\ref{thm:1} and the fact that \newline
$
t_{\mathrm{spectral-profile}}^{\SRW}(\epsilon) \asymp |\log(\epsilon)| \rel^{\rm{SRW}}$ for $\eps<1/|V|$
imply the following corollary.
\begin{corollary}
\label{cor:expanders}
Uniformly in $(\mu,p) \in (0,1]^2 $ and in $G=(V,E) $ we have that for all $k\in \N$
\begin{align*}
 \frac{1}{k}\mixfmupi \left(\frac{1}{|V|^k}\right) \lesssim (\mu p)^{-1} \rel^{\SRW} \log |V| +\frac{1}{\mu}|\log (1-p)|
\\ \asymp {(\mu p)^{-1}} \mixsrw\left(\frac{1}{|V|}\right) +\frac{1}{\mu}|\log (1-p)|. 
\end{align*}
\end{corollary}

\begin{remark}
\rm{
The following quenched statement follows easily from Corollary \ref{cor:expanders}: \\ Let  $t=t(\mu,p,C):=C (\mu p)^{-1} \rel^{\SRW} \log |V| $. Define $d(t,x,\{\eta_s \}_{s \ge 0 })$ to be the total variation distance from $\pi$ of the walk co-ordinate $X_t$ at time $t$ of the full process, given the environment $\{\eta_s \}_{s \ge 0 } $ and given $X_0=x$. Then for some choice of universal constant $C>0$  for all $(\mu,p) \in (0,1] $ we have that
\[\max_{\eta}\P_{\eta}^{(\mu,p)}\left(\max_x d(t,x,\{\eta_s \}_{s \ge 0 }) > 1/|V|^{2} \right)<1/|V|. \]
Indeed, this follows by Markov's inequality and a union bound over $x$. }
\end{remark}

Recall that the \emph{log-Sobolev constant} of a continuous-time Markov chain on a finite state space $\Omega$, with generator $\cL$ and stationary distribution $\pi$ is given by \[c_{\mathrm{LS}}=\inf  \{\sfrac{\EE_{\cL}(h,h)}{\mathrm{Ent}_{\pi}(h^{2})} : h^{2} \in [0,\infty)^{\Omega} \text{ and non-constant}     \} ,\] where $\mathrm{Ent}_{\pi}(f)=\mathbb{E}_{\pi}[f \log(f/\E_\pi[|f|])]$ (with the convention that $0 \log 0=0$).

\begin{theorem}
\label{thm:cLS}
There exist positive constants $c_1$ and $c_2$ so that for every graph $G$ and all values of $(\mu,p)$ we have that
\[c_{\mathrm{LS}}^{\fu,(\mu,p)} \geq c_1\mu  \min \left\{   pc_{\mathrm{LS}}^{\SRW}, \frac{1 }{\log (1/\pi_*)\log(\frac{1}{p(1-p)}) }  \right\},  
\]
where $\pi_*$ above is defined to be $\min_{v \in V}\pi(v)$.  
Moreover, we have that 
\[
\rel^{\rm{full}, (\mu,p)} \leq c_2 \frac{1}{\mu p} \rel^{\rm{SRW}}.
\]
\end{theorem}
 
\begin{remark}\rm{
        We note that for a simple random walk if $h_v(u):=\frac{\1_{u=v}}{\sqrt{\pi(v)}} $, then 
        \[
        c_{\rm{LS}}^{\rm{SRW}} \le \min_{v} \frac{\EE_{\cL}(h_{v} ,h_{v})}{\mathrm{Ent}_{\pi}(h_{v}^{2})}  = \frac{1}{\log (1/\pi_*)},
        \]
        and hence if $p$ is bounded away from $1$, then the lower bound above for the log-Sobolev constant of the full process becomes $c_{3}\mu pc_{\rm{LS}}^{\rm{SRW}}$.
        }
\end{remark}

Let $G=(V,E)$ be an $n$-vertex connected graph. We say that $G$ is \emph{vertex-transitive} if the action of its automorphism group on its vertices is transitive. Denote the volume of a ball of radius $r$ in $G$ by $V(r)$. Denote the diameter of $G$ by $\gamma=\inf\{r:V(r) \ge n \}$. 
Following Diaconis and Saloff-Coste~\cite{moderate}  we say that  $G$ has $(c,a)$-\emph{moderate growth} if $V(r) \ge cn(r/\gamma)^a $ for all $r$. Breuillard and Tointon \cite{BT} proved that for Cayley graphs of fixed degree, this condition is equivalent in some quantitative sense to the simpler condition that $n \le \beta \gamma^{\alpha} $ for some $\alpha,\beta >0$. Tessera and Tointon recently extended this result to vertex-transitive graphs \cite{tessera2019finitary}.
 We note that Diaconis and Saloff-Coste~\cite{moderate} proved that for vertex-transitive graphs of $(c,a)$-moderate growth and of degree $d$
\[
d^{-1} \mixisrw \lesssim_{a,c} \gamma^2 \lesssim_{a,c} \rel^{\SRW} \le \mixisrw.
\]

Denote the percolation cluster of vertex $x$ by $K_x$ and its edge boundary by~$\partial K_x $. We identify the cluster with the vertices lying in it, and denote the cardinalities of  $K_x$ and $\partial K_x $ by $|K_x|$ and $|\partial K_x|$, respectively. 
We also denote $M_p=\pi_p(|\partial K_x||K_{x}|^{2}) \le \pi_p(d|K_{x}|^{3})  $ and $N_p=\pi_p(|K_{x}|)$ (by transitivity these quantities are independent of $x$), where $d$ is the degree and where   $\pi_p(f)=\sum_{\eta} \pi_p(\eta)f(\eta)$ denotes expectation w.r.t.\ $\pi_p$ of $f:\{0,1 \}^E \to \R $.

 \begin{theorem}[Moderate growth vertex-transitive graphs, subcritical regime]
\label{thm:mod}
Let $a,b,c\in \R_+$.  Let $G=(V,E)$ be a connected vertex transitive graph of degree $d$, $(c,a)$-moderate growth and diameter~$\gamma$.
Suppose that $|\log (1-p)|\leq \gamma^2$, $M_p \leq b$ and $N_p\leq \gamma/8$. Then  
\[ 
\mix^{\mathrm{full},(\mu,p),(\infty)} \asymp_{a,b,c,d} (\mu p)^{-1} \rel^{\SRW} \asymp_{a,b,c}\rel^{\mathrm{full},(\mu,p)}.
\]
Moreover, even if $p$ is not subcritical, we still have that $\mix^{\mathrm{full},(\mu,p),(\infty)}\lesssim_{a,c,d} (\mu p)^{-1} \rel^{\SRW}   $.
\end{theorem}
We note that the condition $|\log (1-p)|\leq \gamma^2$ is very mild, and typically follows from either of the two other conditions (provided $b$ is much smaller than $n$). In applications $\gamma \gg 1$ and we think of $b$ as a constant. In this case the conditions $N_p\leq \gamma/8$ and $|\log (1-p)|\leq \gamma^2$ follow from the condition $M_p \leq b$. A more refined statement, with an explicit dependence on $b$ is given in Lemma \ref{lem:mod}.
\begin{theorem}[Hypercube]
\label{thm:hyp}
Let $G$ be the hypercube $\{0,1\}^d$. Then uniformly in $(\mu,p) \in (0,1]^2 $
we have that 
\[
\mixfmupi \lesssim \frac{1}{\mu p} d \log d +\frac{1}{\mu}|\log\left(1-p\right)|. 
\]
\end{theorem}
It is conjectured \cite[p.\ 59]{hypercube} that  when $p=c/d $ for $c>1$ the mixing time of the SRW on the giant component is w.h.p.\ $\Theta(d^2)$, however no polynomial in $d$ upper bound is known for the static case.
\subsection{Remarks and open problems}\label{sec:remarks}

(1) \textbf{Comparison with previous results:} We are able to extend the results of Peres, Stauffer and Steif~\cite{PSSsub} about $\Z_n^d$  to the more general setup of vertex-transitive graphs of moderate growth via a simpler proof, while also eliminating the requirement that $p$ is subcritical. Our results for the supercritical case complement the results on Peres, Sousi and Steif \cite{PSSsuper,PSSexit}. When $\mu$ is of order~1 our Theorem \ref{thm:mod} provides better bounds, but their results are quenched and hence not directly comparable to ours (also, the main result from \cite{PSSsuper} required $p$ to be supercritical in a certain quantitative way, namely that the expected size of the giant component is at least $|V|/2$). We note that we give the first bounds about the mixing time in the critical case.

Biskup and Rodriguez \cite{BR18} study continuous-time random walk on $\mathbb{Z}^d$ for $d \ge 2$ with random, symmetric jump rates which are time dependent, stationary, ergodic and bounded from above, but are not assumed to be positive. Under mild assumptions on the environment (meant to eliminate the possibility that a certain edge has rate 0 for a long period of time) they prove a quenched invariance principle.

 We consider a particular case of this model in remark (9) below, but do so on general vertex-transitive graphs (we also do not require the weights to be bounded). We consider the mixing time of the walk co-ordinate in this model. The main difference is that while we consider the less general (and simpler) case in which the rates of different edges are independent, we provide an upper bound on the mixing time with an explicit dependence on the update rate of the rates.

\noindent (2) \textbf{Refining \eqref{e:hitmain2}:}
It is quite possible that for all regular and all bounded degree graphs one has that $t_{\mathrm{spectral-profile}}^{\SRW} \lesssim \hitsrw $. We strongly believe this to be the case when $G$ is a Cayley graph. Whenever this is the case,~\eqref{e:hitmain2} reads as 
\[
\hitfmup \lesssim  (\mu p)^{-1}\hitsrw.
\]
This follows from Theorem~\ref{thm:1} and Remark~\ref{rem:1-p}.
  We note that this always holds up to a $\log \log |V|$ factor, as by Kozma \cite{kozma} for all graphs   $t_{\mathrm{spectral-profile}}^{\SRW}\lesssim \mixisrw \log \log |V|$, while  $ \mixisrw  \lesssim \hitsrw   $  \cite[Theorem 10.22]{levin}.

\noindent (3)\textbf{Hitting times lower bound for subcritical $p$:} It is reasonable to expect that for $p$ which is subcritical, say in the sense that the second moment $M$ of the size of a cluster is bounded, we have that  $\hitfmup \gtrsim_M  (\mu p)^{-1}\hitsrw $.   (Note that $M$ is increasing in $p$. Hence the $p^{-1}$ term on the r.h.s.\ is meant to capture the dependence on $p$ for $p$ close to 0. Namely, for such $p$ we expect that $\hitfmup \gtrsim  (\mu p)^{-1}\hitsrw$.)

\noindent (4) \textbf{Heuristics for the $(\mu,p)$ dependence:} The following heuristics explain the appearance of the terms $\mu$ and $p$ in our main results when the maximal degree is bounded. Informally, when $p$ is subcritical, viewing the walk when it moves from one cluster to another gives a new random process with small increments, which {(as a discrete-time process)} we expect in many cases to have the same order mixing time as SRW. However, as we now explain,  the time {between  consecutive steps of this process} is typically of order $\sfrac{1}{ \mu p}$. Indeed, the walk co-ordinate of the full process is typically at a cluster $\mathcal{C}$ of size $O(1)$. Using the fact that the degree is bounded, it requires order $\frac{1}{|\mathcal{C}|\mu p}=\Omega(\frac{1}{\mu p})$ time units until an edge adjacent to  $\mathcal{C}$ becomes open.

\noindent (5) \textbf{Direct comparison of mixing times:} It is natural to wonder whether stronger results of the form $\mixfmup \lesssim (\mu p)^{-1}  t_{\mathrm{mix}}^{\SRW} +\mu^{-1}\log |V| $ and $\mixfmupi$ $\lesssim (\mu p)^{-1}  t_{\mathrm{mix}}^{\SRW,(\infty)} +\mu^{-1}|\log (1-p)|$ hold. Alas, it appears that small variants of the examples from \cite{DP,HP,Hsen} can be used to show that this does not hold in general. However, if $G$ is vertex-transitive we believe this should indeed be the case with the implicit constant depending on the degree.

\noindent (6) \textbf{Commute-times:} It follows from our argument that for every $x,y \in V$ the expected time it takes the walk component of the full process to reach $y$ from $x$ and then return to $x$ is at most $O_{\mu}( p^{-1} \E_x^{\SRW}[T_y]+\E_y^{\SRW}[T_x] )$ and also at most $O(\frac{1}{ \mu p}[\E_x^{\SRW}[T_y]+\E_y^{\SRW}[T_x]+t_{\mathrm{spectral-profile}}^{\SRW} ]) $ for $p$ bounded away from~$1$.

\noindent (7) \textbf{Extending Theorem \ref{thm:hyp} to self-products:} The Cartesian product $G_1 \times G_2=(V',E')$ of two graphs $G_i=(V_i,E_i) $ is defined via $V':=V_{1}\times V_2 $ and \[E':=\{\{(v_{1},v_{2}), (u_{1},u_{2})\}:v_1=u_1 \in V_1 \text{ and }v_2u_2 \in E_2, \text{ or vice-versa} \} .\]
For a graph $G=(V,E)$ we denote the $n$-fold self (Cartesian) product of $G$ with itself by $G_{\otimes n}=(V^n,E(G_{\otimes n}))$. That is $G_{\otimes n}=G_{\otimes (n-1)} \times G=G \times \cdots \times G $. Note that the $n$-dim hypercube is the $n$-fold self-product of the complete graph on two vertices with itself. The proof of Theorem \ref{thm:hyp} can easily be extended to show that uniformly in $(\mu,p) \in (0,1]^2$, for  $G_{\otimes n}$ with $G=(V,E)$
we have that 
\[
\mixfmupi \lesssim_{|V|} \frac{1}{\mu p} n\log n +\frac{1}{\mu}|\log\left(1-p\right)|.\]
In fact, this can be derived as an immediate corollary from either \eqref{e:mixmain} (in a similar fashion to the derivation of Theorem \ref{thm:hyp}) or from Theorem \ref{thm:cLS}. 

\noindent (8) \textbf{Allowing $\mu>1$:} We assume throughout that $\mu \in (0,1]$. Our analysis of the case $\mu=1$ can be used almost verbatim to treat $\mu>1$, in which case terms of the form $\frac{1}{\mu}$ should be replaced by $\frac{1+\mu}{\mu}$.

\noindent (9) \textbf{Random rates model:} We now discuss a certain extension of our results to a \emph{random walk on dynamical random rates model}.
Let $\nu$ be some law supported on $\mathbb{R}_+$.
Consider the case that each edge $e$ is updated at rate $\mu$ and when it is updated, it is assigned a random rate $r_e$ with law $\nu$ independently. Given that the current location of the walk co-ordintae is $x$, and that the current environment is $\eta$, where $\eta(e)$ denotes the rate of the edge $e$, the walk co-ordinate moves to vertex $y$ at rate $\eta(xy)$ for all $xy \in E$. Observe that the case that $\nu$ is Bernoulli($p$) gives rise to random walk on dynamical percolation. Let $X \sim \nu$.

We consider the case that the base graph $G=(V,E)$ is vertex-transitive of degree $d$, that $\nu(0) = 1- p$, that $\mathbb{E}[X]<\infty$ and that
$\mathbb{E}[e^{a (X/\mathbb{E}[X])}] \le b$ for some $a,b>0$. Let $\kappa(\mu)$ be the expected time until the walk co-ordinate leaves the origin for the first time in the variant of the model in which for each vertex $v$, at rate $\mu$ all of the rates of the edges incident to $v$ are refreshed simultaneously, independently according to the law $\nu$. Let $M:=C'(a,b)\frac{d\mathbb{E}[X^2 \mid X>0]}{\mathbb{E}[X]}$, where $C'(a,b)$ is a constant depending only on $a,b$, to be determined later. (Note that $M$ implicitly depends on $\nu(0)$.)

In \S\ref{s:randomrates} we extend our analysis to this model and show that (for some appropriate $C'(a,b)$ above) the total variation mixing time of the walk co-ordinate is at most \[ \frac{C(a,b)d (M +\mu)\kappa(\mu \vee M)}{ \mu }t_{\mathrm{spectral-profile}}^{\mathrm{SRW}},\] where $C(a,b)$ depends only on $(a,b)$, and where $A \vee B:=\max \{A,B\}$.

 Moreover, if $\nu$ has a finite support, this is also an upper bound on the total variation mixing time of the full process, whereas the $L_2$ mixing time of the full process is at most
\begin{equation}
\label{e:rrmt}
\frac{C(a,b)d (M +\mu)\kappa(\mu \vee M)}{ \mu }t_{\mathrm{spectral-profile}}^{\mathrm{SRW}}+\frac{Cd|\log \min_{x:x \neq 0}\nu(x) |}{ \mu}.
\end{equation}

\begin{question}
Let $G$ be an infinite connected vertex-transitive graph. Assume that initially the environment is stationary. Let $P_t(o,o) $ and $Q_t(o,o) $ be the return probability to the origin by the (continuous-time) SRW and by the walk co-ordinate of the full process, respectively. Is it the case that for some $C=C(G,p,\mu) \ge 1 $ we have
\[\forall \, t \ge 0, \quad \sfrac{1}{C} P_{Ct}(o,o)\le Q_t(o,o) \le CP_{t/C}(o,o) \, ?  \]
\end{question}

\begin{question}
What is the order of $\mixfmup $ for $G=\Z_n^d$ when $p $ is the critical probability $p_c$ for Bernoulli bond percolation on $\Z_n^d$?
\end{question}
We believe that for all $d$ the order of the mixing time when $p=p_c$ and $\mu=o(1)$ lies strictly between its values in the subcritical and supercritical regimes (for the same $\mu$), and that it has a complicated dependence on $\mu$ and $d$. (It is possible that the dependence on $d$ becomes simple once mean-field behavior kicks in.) This might seem surprising at first glance due to the lack of exceptional times (i.e.\ times at which an infinite cluster exists) for critical dynamical percolation on $\mathbb{Z}^d$ in high dimension \cite{Noinfinite}. We wish to express our immense gratitude to Gabor Pete for relevant discussions.

The \emph{cover time} $\tau_{\mathrm{cov}}^{\mathrm{SRW}}(G)$ of a graph $G=(V,E)$ is the first time by which every vertex $v \in V$ has been visited by SRW on $G$.  Let $t_{\mathrm{cov}}^{\mathrm{SRW}}(G):=\max_{v \in V}\mathbb{E}_v[\tau_{\mathrm{cov}}^{\mathrm{SRW}}(G)]$ be its worst case expectation. We can similarly define  $\tau_{\mathrm{cov}}^{(\mu,p)}(G)$ to be the first time by which every vertex $v \in V$ has been visited by walk co-ordinate of the full process with parameters $(\mu,p)$ (i.e.\ $(v,\eta)$ has been visited for some $\eta$) and $t_{\mathrm{cov}}^{(\mu,p)}(G):=\max_{(v,\eta)} \mathbb{E}_{(v,\eta)}^{(\mu,p)}t_{\mathrm{cov}}^{(\mu,p)}(G)$.
\begin{question}
Is it the case that there exist constants $C(\mu)>0$ such that for all finite graphs $G$ we have that
\[ \forall \, \mu \in (0,1],p \in (0,1], \qquad t_{\mathrm{cov}}^{(\mu,p)}(G) \le \frac{C(\mu)}{p} t_{\mathrm{cov}}^{\mathrm{SRW}}(G) \text{ ?} \]
\end{question} 
We comment that if $G=(H,E)$ is a Cayley graph of an Abelian group $H$, then the auxiliary process from \S~\ref{s:auxchain} is reversible (see Remark \ref{r:abelian}). In~\cite{DLP} the authors give a general comparison principle for cover times (of two reversible Markov chains on the same state space, provided the effective resistances of one chain are pairwise smaller than some constant multiple of those of the other). 
Thus using our comparison of the transition probabilities between the auxiliary chain and the SRW from Lemma \ref{lem:PauxPsrw} (together with Theorem \ref{lem:commutetime} which relates this to a comparison of effective resistances), it follows from \cite{DLP} that $ t_{\mathrm{cov}}^{(\mu,p)}(G) \le \frac{C(\mu)}{p}t_{\mathrm{cov}}(G)$, by considering the cover time of $G$ for the auxiliary walk. We omit the details.
\section{Overview of our approach}

We believe that our approach may be relevant for other models of random walks on evolving graphs, in situations in which the mixing time of the environment is smaller than that of the walk co-ordinate.
We first explain why we may concentrate on the case that $\mu=1$. Let
$\cL_{(\mu,p)}$ be the infinitesimal generator of the full process with edge probability $p$ and update rate $\mu$. One can readily see that if $\mu_1<\mu_2$ then the corresponding rates satisfy \[\forall \, x,x' \in V,\, \eta,\eta' \in \{0,1\}^E, \quad \cL_{(\mu_2,p)}((x,\eta),(x',\eta')) \le \sfrac{\mu_2}{\mu_1} \cL_{(\mu_1,p)}((x,\eta),(x',\eta')). \]
Using this and some general theory we are able to transfer hitting time estimates from $\mu=1$  to ones for $\mu<1$, at a cost of a $\mu^{-1}$  multiplicative term and an additive term of order of $\mixfmup $. As for the mixing time, we obtain an upper bound on the spectral-profile bound when $\mu=1$ by a certain comparison with SRW. Again, general theory allows us to then translate this into a bound for general $\mu<1$ at a cost of a $\mu^{-1}$  multiplicative term. 

Now consider the case that $\mu=1$. As we now explain, we may also consider the case  that $\eta_0 \sim \pi_p $. It is tempting to argue that it suffices to wait until every edge is updated once, which takes order $\log |E| $ time units, and then the environment is stationary. However, since the walk co-ordinate is dependent on the environment some difficulties arise when trying to formulate this. Using the fact that $\mu = 1$ we show that at some stopping time slightly larger than the first time at which each edge is updated at least once, the environment is stationary and is independent of the position of the walk. This explains why when bounding hitting times we may assume that $\eta_0 \sim \pi_p $.

 We define a sequence of increasing stopping times (w.r.t.\ an enlarged filtration) $(\tau_i: i \in \Z_+ )$ such that the following hold:
\begin{itemize}
\item[(1)]  $\eta_{\tau_i} \sim \pi_p $ and is independent of $X_{\tau_i} $ for all $i$. 
\item[(2)] $\tau_0=0$ and $\tau_{i+1}-\tau_i $ are i.i.d.\ such that $\E\left[ e^{\delta \tau_1}\right]<\infty $ for some $\delta>0$. 
\item[(3)] $Y_i=X_{\tau_i} $ is a Markov chain with the same stationary distribution as SRW. We call each such $\tau_i$ a regeneration time. 
\end{itemize}
By construction, we will have  for all $i \in \N$ that  all of the edges examined by the walk co-ordinate (i.e.\ that the walk co-ordinate attempted to cross) at some time during $(\tau_{i-1},\tau_i)$ have been refreshed since the last time in  $(\tau_{i-1},\tau_i)$  at which they were examined. This, along with the assumption that $\eta_0 \sim \pi_p $, imply the crucial property (1) above.     

While we do not believe the auxiliary chain $(Y_i)_{i \in \Z_+}$ to be reversible in general, we show that $\pi$ is its stationary distribution, and that the transition matrix $Q$ of its additive-symmetrization satisfies that $Q(u,v) \gtrsim \sfrac{1}{p \deg u} $ which turns out to be sufficient in order to obtain a comparison of its hitting times (more precisely, of its commute times) and of its spectral-profile with that of SRW (at a price of an $O( \sfrac{1}{p}) $ factor). To turn these into upper bounds on expected hitting times  for  $(Y_i)_{i \in \Z_+}$ (rather than for its additive symmetrisation) we use the fact that the commute times for the symmetrisation are always as large as they are for the original chain.   As the chain $(Y_i)_{i \in \Z_+}$ is constructed by viewing the chain $(X_t)_{t \in \R_+ }$ at a nice sequence of stopping times, we then easily translate hitting time bounds for $Y$ into ones for  $X$ (for hitting times of the form $T_{y \times \{0,1\}^E }$).

For the mixing time, we derive a comparison of the spectral profile bound  $t_{\mathrm{spectral-profile}}^{\fu,(1,p)}$ for the full process with parameters $(1,p)$ with that of SRW  $t_{\mathrm{spectral-profile}}^{\SRW}$. The comparison method is a standard method for comparing analytic quantities that have an extremal characterization involving the Dirichlet form, between two Markov chains. Despite the fact that the auxiliary chain may be non-reversible, using the comparison method we are able to effortlessly compare the spectral-profiles of SRW with that of the auxiliary chain. We wish to use the auxiliary chain as a link to relate the SRW with the full process. 

A major obstacle in doing so is that the former has state space $V$, while the latter $V \times \{0,1\}^V $. 
 Unfortunately, the comparison method requires the two considered chains to have the same state space, or at least similar state spaces (e.g.,   random walks on two quasi-isometric graphs might have different state spaces, but with more effort can still be compared). We are unaware of  any previous works establishing a comparison argument between Markov chains with very different state spaces. 

Our solution to this difficulty hinges on a probabilistic interpretation of     the spectral profile  in terms of the rate of exponential decay of the tail of hitting times starting from the stationary distribution of the chain, which we review in Section~\ref{s:profiles}. This interpretation allows us to use the the auxiliary chain  as a link between the full process and the SRW.

 For each set $A \subseteq V \times \{0,1\}^E   $ we construct a certain set $B \subseteq V $ such that $\pi(B^{c})\lesssim \pi_{\fu,p}(A^{c}) $ and show that the rate of exponential decay of the tail of $T_A$ for the full process can be controlled via that of $T_B$ for the aforementioned auxiliary chain $Y$. The set $B$ is defined to be the set of all vertices $v \in V$ such that $\{w \in \{0,1\}^E : (v,w) \in A \} $ has  $\pi_p$ probability at least $1/4$. Let $C \ge 1 $ be some absolute constant. Let \[\gamma_{A}:=\sup \{a:\mathbb{E}_{\pi}[e^{aT_{D}}]< \infty  \text{ for all }D \subset V \text{ such that }\pi(D^{c}) \le C \pi_{\fu,p}(A^{c}) \}, \] where the expectation is taken w.r.t.\ SRW. The connection of this quantity to the spectral-profile of SRW is explained in Section~\ref{s:profiles}. It follows from this connection that to prove Theorem \ref{thm:1} it suffices to show that for some absolute constant $c>0$ we have that $\E_{\pi_{\fu,p}}[\exp(cp\gamma_{A}T_{A})]< \infty $ (recall that we consider $\mu=1$). 

 Let $t=\frac{C'}{\gamma_{A}  p }$.
Loosely speaking, we show that for some choice of $(C,C')$,  for all $i \in \N$,     the full process started from its stationary distribution has a regeneration time during the time interval  $[2it,(2i+1)t]$, with probability bounded from below. Moreover,  at each such  regeneration time, with probability bounded from below, the full process is in $A$. Furthermore, this holds even conditioned on the information gathered at the previous time intervals  $[2jt,(2j+1)t]$, where for each such time interval we expose if there was a regeneration time during it, and if there was one, we take the first one and expose whether the walk is in $A$ at that time. 

(In the actual proof we do not restrict to one regeneration time per time interval, although we could have done so; for technical reasons we make several other small modifications in the proof to what is written above.)

This is achieved by treating the walk co-ordinate and the environment separately. For the walk co-ordinate we use the comparison method to compare the auxiliary chain with the SRW. By comparing their spectral-profiles, and applying a certain inductive argument (which we sketch below) we are able to show that at each time interval  $[2it,(2i+1)t]$ the chain has probability bounded from below to visit $B$, even when conditioning on the information exposed during the previous time intervals. This gives us a sequence of visits to $B$ at regeneration times by the walk co-ordinate, which are separated by at least $t$ time units.   

For the environment, note that by definition, at every visit of the walk co-ordinate to $B$, if the environment is distributed according to $\pi_p$ and independent of the walk, the full process would have probability at least $1/4$ of being in $A$.  By property (1) above, the distribution of the environment is indeed stationary at all regeneration times and in particular at the regeneration times at which the walk visits $B$. However, (generally) if $(Y_i)$ is a stationary Markov chain, and for each $j$, $D_j$ is an event which is determined by $Y_j$, then given the values of the indicators of $D_1,\ldots,D_i$ we no longer have that $Y_{i+1}$ has the same law as its unconditional law (i.e., its law might differ from the invariant distribution). In particular, if at some of the previous visits to $B$ at regeneration times we expose if the full process is in $A$, this affects the law of the environment at the current regeneration time, making it no longer stationary. It turns out that the fact that we consider a sequence of such visits which are separated by an amount of time units which is larger than the relaxation-time of the environment allows us to control such dependencies. We now review the general principle from which this follows.

For a Markov chain $(W_t)_{t \in \R_+}$ on a finite state space,  reversible w.r.t.\ a distribution $\hat \pi$, starting from its invariant distribution, if $s $ is a large multiple of its relaxation-time, then samples of the chain at times $0,s,2s,\ldots $ are uncorrelated (e.g.\ \cite[\S12.7]{levin}). In fact, if $D_0,D_1,D_2,\ldots $ is a sequence of sets such that $\inf\{ \min\{ \hat \pi(D_i), \hat \pi(D_i^{c})\}:i \ge 0 \}=:\delta>0 $, then as we explain below, using $L_2$ considerations and an inductive argument, one can show that if $s \ge C_{1} \rel \log (1/\delta) $ for some absolute constant $C_1$ then a.s. \[\E[ \1_{W_{is} \in D_i} \mid ( \1_{W_{js} \in D_j^c}:j<i)   ] \ge \delta/2. \]
Note that this statement is obvious is we replace $C \rel \log (1/\delta)$ by the $\delta/2$ mixing time.   As two more complicated variants of this inductive argument will be used in the proof of Theorem \ref{thm:1}, we now give a sketch of the proof of this fact, in order to emphasise the main idea behind the proof.

 Let  $\mu $ be a  distribution with $\|\mu- \hat \pi\|_{2,\hat \pi} \le \delta $.   Let $\xi_i $ be the indicator of $W_{is} \in D_i $. Then for all $i$ and all $a_0,a_{1}\ldots\in \{0,1\} $, such that $\P_{ \mu}[E_{\ell}  ]>0$ for  all $\ell $, where $E_{\ell}=\{\xi_j=a_j \text{ for all }j \le \ell \}$,   we have that  \[\P_{\mu}[W_{is} \in D_i \mid E_{i-1}  ]   \in \left[\hat \pi(D_i)-\frac {\delta}{2} ,\hat \pi(D_i)+\frac {\delta}{2} \right]. \]  

This is proven by proving by induction that for all $i \in \N $ \[\Upsilon_{i}:=\|\P_{\mu}(W_{is} \in \cdot \mid E_{i} ) \|_{2,\hat \pi}^2 \le 4 \frac{ \delta^{2} +1 }{\min\{\hat \pi(D_i),\hat \pi(D_i^{c}) \}^{2}},  \] and that \[|\P_{\mu}(W_{(i+1)s} \in D_{i+1} \mid E_{i} )-\hat \pi (D_{i+1})| \le \|\P_{\mu}(W_{(i+1)s} \in \cdot \mid E_{i} )-\hat \pi \|_{2,\hat \pi} =: \varrho_i\le  \delta.\] The induction step follows by combining the following simple observations (in proving the induction step, first use (3) then (2) and then (1)):
\begin{itemize}

\item[(1)] By the definition of the total-variation distance ($\|\nu-\nu'\|_{\mathrm{TV}}=\max_D$ \\ $ |\nu(D)-\nu'(D)|$), and the fact that $2\|\nu-\nu'\|_{\mathrm{TV}} \le \|\nu-\nu'\|_{2,\pi} $, we have \[ |\P_{\mu}(W_{(i+1)s} \in D_{i+1} \mid E_{i} )-\hat \pi (D_{i+1})| \le \half \varrho_i.  \]

\item[(2)] By the Poincar\'e inequality
\[\rho_i \le \Upsilon_{i}e^{-s/\rel } =\delta^{C} \Upsilon_{i}.  \]

\item[(3)] For any distribution $\nu$ and any set $D$ we have  $\|\nu_{D}- \hat \pi\|_{2,\hat \pi}^2 +1\le \frac{\|\nu-\hat \pi\|_{2,\hat \pi}^2+1 }{\nu(D)^{2}}  $, where $\nu_{D}(b ):=\frac{\nu(b)\mathbf{1}_{\{b \in D\}}}{\nu(D)} $ is $\nu$ conditioned on $D$ (see \eqref{e:hatnuL2}). Hence by  (1) $\Upsilon_{i}^2 \le \frac{\varrho_{i-1}^{2}+1 }{\min\{\hat \pi(D_i)- \half \varrho_{i-1},\hat \pi(D_i^{c}) - \half \varrho_{i-1}\}^{2}}$.
\end{itemize}
To argue that  at each time interval  $[2it,(2i+1)t]$ the chain has probability bounded from below to visit $B$, even when conditioning on the information exposed during the previous time intervals, we use a   more complicated version of this inductive argument sketched below, in which $\frac{1}{p\gamma_A}$ and the spectral-profile bound on the decay of the $L_2$ distance from equilibrium play the roles of the relaxation-time and the Poincar\'e inequality above, respectively.

 Consider a sequence of visits to the set $B$ by the walk co-ordinate of the full process at regeneration times, say to vertices $b_1,b_2,\ldots$, at times that are separated apart by $t \ge C_2$ time units, for some large absolute constant $C_2$. At the $i$-th visit, if the environment is in $D_i:=\{w \in \{0,1\}^V :(b_i,w) \in A \} $, then the full process visits $A$ at that time. Let $\xi_i $ be the indicator of the event that the environment is in $D_i$ at this time.
 
We start with a stationary environment, since we start from the stationary distribution on the full process, as we are seeking to show that  $\E_{\pi_{\fu,p}}[\exp(cp\gamma_{A}T_{A})]< \infty $. Since $t$ is large in terms of the relaxation-time of the environment, which is $O(1)$, in light of the above discussion it is intuitive that 
 $\xi_1,\xi_2,\ldots$ is approximately a sequence of independent Bernoulli trials. Making this intuition precise turns out to be quite subtle, and is among the main technical challenges in the proof of Theorem \ref{thm:1}. Indeed, a major obstacle is the fact that the sets $D_1,D_2,\ldots$ are random and depend on the walk co-ordinate, and thus also on the environment.

\section{Auxiliary chain}
\label{s:auxchain}

Recall that $G=(V,E)$ with $V$ the set of vertices and $E$ the set of edges. We fix an ordering of the edges $E=\{e_1,e_2,\ldots, e_n\}$.

For every edge $e_i$ we create an infinite number of copies denoted $e_{i,1},e_{i,2},\ldots$. We emphasize that the copies will not be considered as edges of $E$. We start $(X_0,\eta_0)\sim \pi\times \pi_p$. For every time $t$ we now define a set of ``infected edges'' $R_t$ as follows: $R_0=\emptyset$ and if $R_{t-}=A$ and the exponential clock of $X$ rings at time $t$, we then add to $R_t$ the edge that $X$ examines to cross at time $t$. If this edge already exists in $A$, then we add its lowest numbered copy that is not in $A$. We next assign an ordering to the edges in $R_t$ using the ordering of the edges of~$E$ in the following way: assign label $1$ to the edge (or copy) contained in $R_t$ with the lowest label in the ordering of $E$. If both the edge and some of its copies are contained in $R_{t}$, then we assign label $1$ to the edge and  then we give the next label to the lowest numbered copy and so on.
If there are only copies of that edge, then assign label $1$ to the lowest numbered copy of this edge and label $2$ to the second lowest numbered copy and so on, until we exhaust all the copies of the edge. Then we continue in the same way, by finding the second edge (or copy of it) from $E$ with the second lowest label and assign to it the next label and so on. 
(We note that any ordering of $R_t$ would work, but we choose to specify one to make the construction clearer.) 

When $R_{t-}=A$, then assign exponential clocks of rate $\mu$ to the edges of $E$  that are not in $A$ and we also generate an exponential clock of rate $|A|\mu$. When a clock of an edge in $E\setminus A$ rings, refresh the state of the edge to open with probability $p$ and closed with probability $1-p$. If the exponential clock of rate $|A|\mu$ rings at time $t$, choose an index from $\{1,\ldots, |A|\}$ uniformly at random and remove the edge with this label from the set $R_{t-}$ in order to obtain the set $R_t$. If the edge that was chosen to be removed was an edge of $E$, then we also refresh its state. 

With this construction $(X,\eta)$ has the correct transition rates. This construction also enables us to couple different systems by keeping their infected sets of the same size. More specifically, suppose that $(X^1,\eta^1), \ldots, (X^m,\eta^m)$ are such that at time $0$ they are i.i.d.\ and distributed according to~$\pi\times \pi_p$ and take $R^1_0=\emptyset, \ldots, R^m_0=\emptyset$. Then we couple them all together by using the same exponential clocks of rate $1$ for the walk components and in order to remove edges from the sets $R_{t-}^i$ we use the same exponential clocks of rate $|R_{t-}^i|\mu$ and choose the same uniform number from $\{1,\ldots, R_{t-}^i\}$. For the edges not in $R_{t-}^i$, we assign independent exponential clocks to the different systems. Finally, when a clock rings (either of the walk or of an edge of the set $R$), the vertex to which the walk components jump or the new states of the edges chosen are independent for the different systems.
With this coupling, indeed the sizes of the sets $R_t^i$ remain equal throughout for different $i$'s.

We now want to consider the jump chain of the full process. This is defined to be the full process $(X,\eta)$ observed at the sequence of jump times, which are the points of a Poisson process of parameter~$1+\mu n$. 

The following lemma (whose proof follows immediately by induction on $k$) shows that with the coupling described above the jump chains of the full processes are independent at every discrete time step.

\begin{lemma}\label{lem:adversary}
        Suppose that $(X_k^1,\eta_k^1)_{k\in \N}, \ldots, (X_k^m,\eta_k^m)_{k\in \N}$ start independently according to $\pi\times \pi_p$ and suppose that an adversary prescribes in advance at which discrete time steps the $X$ or the $\eta$ coordinate will make a jump (i.e.\ the jumps of all of the systems occur simultaneously, and either in all of them, the walk co-ordinate attempts a move, or in all of them the environment is updated). When the adversary chooses the $\eta$ coordinate to update, he also prescribes which edge is going to be updated in each system (not necessarily the same edge!).  At times at which the $X$ coordinate is chosen, where the $X$ coordinate jumps to is independent for the different chains. At times when the environment is chosen, the states of the chosen edges are also i.i.d.\ and become open with probability $p$ and closed with probability $1-p$. Then for all $k$ we have that $(X_k^1,\eta_k^1), \ldots, (X_k^m,\eta_k^m)$ are independent. 
\end{lemma}

\begin{lemma}\label{lem:couplingmarkov}
Consider a continuous time Markov chain with generator $Q=e^{\lambda (P-I)}$ for $\lambda >0$. 
        Let $X^1,\ldots, X^m$ be continuous time Markov chains with generator $Q$ which start from i.i.d.\ states distributed according to the invariant distribution $\pi$. Suppose that their exponential clocks are coupled in some way and at every discrete time step the jump chains (that evolve according to the transition matrix $P$) are independent. Then for every $t$ we have 
        \[
        \prstart{X_t^1=x_1,\ldots, X_t^m=x_m}{} = \prod_{i=1}^{m} \pi(x_i).
        \]
\end{lemma}

\begin{proof}[\bf Proof]
        
        Let $N^1, \ldots, N^m$ be the Poisson processes of rate $\lambda$ associated to each Markov chain $X^1,\ldots, X^m$ that are coupled as in the statement of the lemma. Let $Y^1,\ldots, Y^m$ be the independent jump chains that start at time $0$ independently according to $\pi$. Then we have 
        \begin{align*}
                &\prstart{X_t^1=x_1,\ldots, X_t^m=x_m}{} \\&= \sum_{k_1,\ldots, k_m} \prstart{Y_{k_1}^1=x_1,\ldots, Y_{k_m}^m=x_m, N^1_t=k_1,\ldots, N_t^m=k_m}{}\\
                &= \sum_{k_1,\ldots, k_m} \prstart{Y_{k_1}^1=x_1,\ldots, Y_{k_m}^m=x_m}{}\pr{N^1_t=k_1,\ldots, N_t^m=k_m} \\&= \sum_{k_1,\ldots, k_m} \pi(x_1)\cdots \pi(x_m)\pr{N^1_t=k_1,\ldots, N_t^m=k_m} =  \pi(x_1)\cdots \pi(x_m)
        \end{align*}
        and this completes the proof.
\end{proof}

%
 
\begin{definition}\label{def:aux}\rm{ 
We define the auxiliary chain $Y$ starting from $x_0$ as follows:
start $\eta_0\sim \pi_p$ and $X_0=x_0$. Set $R_0=\emptyset$ and for every $t$ consider the set of infected edges $R_t$ as defined above. We define the regeneration times by letting $\tau_0=0$ and for every $i\geq 0$ we set 
\[
\tau_{i+1}=\inf\{t\geq \tau_i+S_i: R_t=\emptyset\},
\]
where $\tau_i+S_i$ is the first time after time $\tau_i$ that $R_t$ becomes nonempty. The auxiliary chain is defined to be the discrete time chain given by $Y_i= X_{\tau_i}$ for all $i$.
}
\end{definition}

\begin{remark}
        \rm{
Note that the $(S_i)$'s are i.i.d.\ having the exponential distribution with parameter $1$ and are independent of $(X_{\tau_i})_i$. }
\end{remark}

\begin{lemma}\label{lem:prelimforproof}
        We have that $(\tau_i-\tau_{i-1})_{i\geq 1}$ are i.i.d., have mean $e^{1/\mu}$ and have exponential tails. Moreover, the process~$(X_s, \eta_{s}, R_s)_s$ is positive recurrent. 
\end{lemma}

\begin{proof}[\bf Proof]

The first claim follows from the fact that $|R_t|$ evolves as a birth and death chain with transition rates $q(i,i+1)=1$ and $q(i,i-1)=\mu i$. 

We now prove that the chain $(X,\eta,R)$ is positive recurrent. It suffices to prove that the state~$(x,\eta,\emptyset)$ with $x\in V$ and $\eta\in \{0,1\}^E$ is positive recurrent. Let $T$ be the first return time to $(x,\eta,\emptyset)$. Consider the chain $Z_i=(X_{\tau_i},\eta_{\tau_i})$. This is clearly irreducible and positive recurrent, since it takes values in a finite state space. Let $T_Z$ be the time it takes for $Z$ to return to $(x,\eta)$. Then $T_Z$ has exponential tails and we have 
\[
\E[T]\leq \sum_{i=1}^{\infty}\E\left[(\tau_i-\tau_{i-1})\1(T_Z\geq i)\right].
\]
Using Cauchy-Schwarz and the exponential tails of both $T_Z$ and $(\tau_i-\tau_{i-1})$ proves the result.      
\end{proof}

\begin{lemma}\label{lem:inv}
        The invariant distribution of $Y$ is $\pi$.
\end{lemma}

\begin{proof}[\bf Proof]

In order to prove this result, we let $(X^1,\eta^1),\ldots, (X^m,\eta^m)$ be $m$ systems started independently from stationarity and coupled in the way described above using the infected sets, so that at every time $t$ they are independent and their infected sets all have the same size $|R_t|$. Then by Lemmas~\ref{lem:adversary} and~\ref{lem:couplingmarkov}
we get that $(X_t^1,\eta_t^1),\ldots, (X_t^m,\eta_t^m)$ are independent at every time $t$ and are distributed according to $\pi\times \pi_p$.

Fix $x\in V$ and define  
\[
S_m(t) = \frac{1}{m}\sum_{i=1}^{m} \1(X_t^i=x).
\]
Then $\E[S_m(t)] = \pi(x)$ and by the strong law of large numbers we get that as $m\to \infty$
\[
S_m(t) \to \pi(x) \text{ almost surely}.
\]
Using this and dominated convergence gives that as $m\to\infty$
\[
\econd{S_m(t)}{|R_t|=0} = \frac{\E\left[S_m(t)\1(|R_t|=0)\right]}{\pr{|R_t|=0}} \to \frac{\E\left[\pi(x) \1(|R_t|=0)\right]}{\pr{|R_t|=0}} = \pi(x).
\]
Using that
\[
\E\left[S_m(t) \1(|R_t|=0)\right] = \pr{|R_t|=0} \econd{S_m(t)}{|R_t|=0}
\]
and taking the limit as $m\to\infty$ we now obtain 
\begin{align}\label{eq:indep}
\pr{X_t^1=x, |R_t|=0} = \pr{|R_t|=0} \pi(x).
\end{align}
If we now take the limit as $t\to\infty$, then $(X_t^1,\eta_t^1, R_t)$ will converge to its invariant distribution $\mu$, which exists by Lemma~\ref{lem:prelimforproof}. So we obtain
\[
\mu(x,\{0,1\}^E,\emptyset) = \nu(\emptyset) \pi(x),
\]
where $\nu$ is the marginal for $R$ of the invariant distribution $\mu$.

In order to find the invariant distribution of $Y$ we use the ergodic theorem. First of all it is easy to see that $Y$ is an irreducible Markov chain. Let $b$ be its invariant distribution. Then by the ergodic theorem we will have that almost surely as $n\to\infty$
\[
\frac{1}{n}\sum_{i=1}^{n} \1(Y_i=x) \to b(x).
\]
Using the ergodic theorem for $(X,\eta, R)$ we get that almost surely as $t\to\infty$ 
\[
\frac{1}{t} \int_0^t \1(X_s=x, R_s=\emptyset)\,ds \to \mu(x,\{0,1\}^E,\emptyset) = \nu(\emptyset) \pi(x).
\]
Let $N_t=\sup\{i: \tau_i\leq t\}$ and let $(S_i)$ be i.i.d.\ exponential random variables of parameter $1$. Then 
\begin{align*}
 \frac{1}{t} \sum_{i=1}^{N_t-1} \1(X_{\tau_i}=x) S_i \le     \frac{1}{t} \int_0^t \1(X_s=x, R_s=\emptyset)\,ds \le \frac{1}{t} \sum_{i=1}^{N_t} \1(X_{\tau_i}=x) S_i.
\end{align*}    
Using the above convergence this now implies that almost surely as $t\to\infty$
\begin{align*}
        \frac{N_t}{t}\cdot \frac{1}{N_t}\sum_{i=1}^{N_t} \1(X_{\tau_i}=x) S_i \to \nu(\emptyset) \pi(x).
\end{align*}
Since $N_t$ is a renewal process with inter-arrival times having mean $1/\nu(\emptyset)$ we get almost surely
\begin{align*}
        \frac{1}{N_t}\sum_{i=1}^{N_t} \1(X_{\tau_i}=x) S_i \to \pi(x) \quad \text{ as } t\to \infty. 
\end{align*}
Now we notice that since $N_t$ can increase by $1$ at every jump, we get that the sequence $a_n=n$ is a subsequence of $(N_t)_t$, and hence almost surely
\begin{align*}
        \frac{1}{n} \sum_{i=1}^{n} \1(X_{\tau_i}=x) S_i \to \pi(x) \quad \text{ as } n\to \infty. 
\end{align*}
Now to conclude we write $L(n) = \sum_{i=1}^{n} \1(X_{\tau_i}=x)$ and by relabelling $S_i$ using the independence between $(X_{\tau_i})_i$ and $(S_i)_i$ we obtain almost surely
\begin{align*}
        \frac{1}{n} \sum_{i=1}^{L(n)} S_i \to \pi(x) \quad \text{ as } n\to \infty. 
\end{align*}
Using the strong law of large numbers for the i.i.d.\ sequence $(S_i)$ and that $L(n)\to \infty$ as $n\to\infty$ we finally get almost surely
\begin{align*}
        \frac{L(n)}{n} \to \pi(x) \quad \text{ as } n\to \infty.
\end{align*}
This shows that $b(x)=\pi(x)$ for all $x$ and concludes the proof.
\end{proof}

\begin{remark}
\label{r:abelian}
\rm{Consider a graph $G=(V,E)$ with the property that for all $x,y \in V$ there is an automorphism $\phi$ of $G$ such that $\phi(x)=y$ and $\phi(y)=x$. Any Cayley graph of an Abelian group has this property. It is not hard to verify that graphs with this property satisfy that the transition matrix of the auxiliary chain is symmetric and hence reversible w.r.t.\ the uniform distribution.}   
\end{remark}

We end this section with a lemma on the holding probabilities for the auxiliary chain that will be used later in the paper. 

\begin{lemma}
\label{lem:holdingprobforaux}
Let $Y$ be the auxiliary chain with parameters $\mu=1$ and $p\in (0,1)$. Then for all~$x\in V$ we have that  $P_{\mathrm{aux}} (x,x) \ge 1 - 2e^2p $ and  $(1-p)/2 \le P_{\mathrm{aux}} (x,x) \le 1 - \sfrac{p}{2} $,  where  $P_{\mathrm{aux}}$ is the transition matrix of $Y$. 
\end{lemma}

\begin{proof}[\bf Proof]
The claim that  $(1-p)/2 \le P_{\mathrm{aux}} (x,x) \le 1 - p/2 $ follows by considering the case that the first edge to be examined by the walk after a regeneration time is closed (respectively, open), and that this edge is then refreshed before the walk attempts to make further jumps.

We now prove $\min_x P_{\mathrm{aux}} (x,x) \ge 1-2e^2p $. In fact, we show that $\alpha= \min_x \alpha_{x} \ge 1-2e^2p $, where~$\alpha_x$ is the probability that after a regeneration time at $x$ the walk does not leave $x$ before the next regeneration time.  We first observe that after a regeneration time at $x$, until the walk co-ordinate leaves $x$ for the first time, every edge it examines is open with probability at most $p$. To see this, by the strong Markov property, we may consider starting at time 0 from a stationary environment with the walk co-ordinate at $x$. Finally, given that the edges $e_1,\ldots,e_k$ were examined thus far, at times $t_1<\cdots<t_k$ and were all closed, the probability that $e_{k+1}$ is open at time $t_{k+1}>t_k$ is exactly $p$ if $e_{k+1}\notin \{ e_1,\ldots,e_k\} $ and is $p(1-e^{t_i-t_{k+1}})<p $ if $e_{k+1}=e_i$ and  $e_{k+1}\notin \{ e_{i+1},\ldots,e_k\} $.  

We now consider a birth and death chain with an additional point added to it. Its state space is $\Z_+ \cup \{ \infty \} $. Its transition matrix $Q$ is given by $Q(0,1)=1-p=1-Q(0,\infty) $, while  for $i \in \N$ we have $Q(i,i-1)=\frac{i}{i+1} $, $Q(i,i+1)=\frac{(1-p)}{i+1} $ and $Q(i,\infty)=\frac{p}{i+1}$. Finally, we set $Q(\infty,\infty)=1$. It is easy to see that
\[
\alpha \ge \mathbb{P}_0^Q(T_0^+<T_{\infty}),
\]
where $T_0^+$ stands for the first return to $0$ and $T_\infty$ for the first hitting time of $\infty$ for the chain with matrix $Q$. We have that 
\begin{align}\label{eq:boundontoq}
\mathbb{P}_0^Q(T_0^+<T_{\infty}) = (1-p) \mathbb{P}^Q_1(T_0<T_{\infty}).
\end{align}
Consider now a birth and death chain on $\Z_+$ with transition matrix $W(x,y) = \frac{Q(x,y)}{1-Q(x,\infty)}$. Let $N$ be a geometric random variable of success probability $p$ independent of the Markov chain $W$. Since the probability of jumping to $\infty$ for the Markov chain $Q$ is at most $p$ at every step, we can lower bound $\mathbb{P}^Q_1(T_0<T_\infty)$ by the probability that the Markov chain $W$ hits $0$ starting from $1$ before time $N$, i.e.
\[
\mathbb{P}_1^Q(T_0<T_\infty) \geq \mathbb{P}^W_1(T_0<N).
\]
We now turn to lower bound $\mathbb{P}^W_1(T_0<N)$. Since from now on we will only be working with the chain $W$ we drop the letter $W$ from the notation. Let $X$ be the birth and death chain with matrix~$W$. We let $Z=\sum_{i=0}^{N}\1(X_i=0)$. Then
\begin{align}\label{eq:ton}
\prstart{T_0<N}{1} = \prstart{Z>0}{1} = \frac{\estart{Z}{1}}{\escond{Z}{Z>0}{1}} = \frac{\estart{Z}{1}}{\estart{Z}{0}},
\end{align}
where for the last equality we used the memoryless property of $N$.
Using the independence of $N$ and $X$ and that $W(0,1)=1$, we get
\begin{align}\label{eq:equalities}
        \estart{Z}{1} = \sum_{i=1}^{\infty} \prstart{X_i=0}{1} (1-p)^{i-1} \quad \text{ and } \quad \estart{Z}{0} = 1+ \estart{Z}{1}.
\end{align}
We notice that $W(1,0)=\frac{1}{2-p}$, and so  $\prstart{X_i=0}{1} = \prstart{X_{i-1}=1}{1}/(2-p)$.
Therefore, this gives 
\begin{align*}
        \estart{Z}{1}= \frac{1}{2-p} \sum_{i=0}^{\infty} \prstart{X_i=1}{1}(1-p)^i = \frac{1}{2-p} \sum_{i=0}^{\infty} \prstart{X_{2i}=1}{1}(1-p)^{2i}. 
\end{align*}
Solving the detailed balance equations it is straightforward to see that $X$ has an invariant distribution satisfying $\nu(1) \geq e^{-2}$. Using that $\prstart{X_{2i}=1}{1}$ is decreasing as a function of $i$ and converges to $2\nu(1)$ as $i\to \infty$, we obtain
\begin{align*}
        \estart{Z}{1}\geq  \frac{2\nu(1)}{2-p}\sum_{i=0}^{\infty} (1-p)^{2i}=\frac{2\nu(1)}{p(2-p)^2}\geq \frac{1}{2pe^{2}}.
\end{align*}
Substituting this bound and~\eqref{eq:equalities} into~\eqref{eq:ton} we deduce
\[
\prstart{T_0<N}{1} \geq \frac{1}{2e^2p + 1-p} \geq 1- (2e^2-1)p.
\]
Substituting this into~\eqref{eq:boundontoq} we finally get $\alpha \geq 1-2e^2 p$ and this concludes the proof.
\end{proof}

\section{Comparison of hitting times}
\label{s:comparison}

In this section we prove Theorem~\ref{thm:2}. Before doing so, we review some results about hitting and commute times. 

Let $(Y_k)_{k\in \N} $ be an irreducible discrete time Markov chain on a finite state space $\Omega$ with transition matrix $P$ and stationary distribution~$\pi$. For $a\in \Omega$ we write 
\[
T_a=\inf\{t\geq 0: Y_t = a\} \quad  \text{ and } \quad T_a^+ = \inf\{t\geq 1: Y_t=a\}.
\]
For distinct states $a,b\in \Omega$ we write $T_{ab}$ for the commute time between 
$a$ and $b$ (for the chain starting at $a$), i.e.
\[
T_{ba} = \inf\{ t>T_b: Y_t=a\}.
\]
We write $P^*$ for the time-reversal of $P$, i.e.\ it is the transition matrix given by 
\[
\pi(x) P^*(x,y) = \pi(y) P(y,x).
\]
Let $\lambda(A)$ be the smallest eigenvalue of $Q_A$, the generator of the chain killed upon exiting $A$.
Writing $Q = \gamma (P-I)$ for some transition matrix $P$ and $\gamma>0$ and using the Perron Frobenius Theorem we have that \cite[Chapter~3]{aldous}
\begin{align}\label{eq:lambdaa}
\lambda(A) = \min\left\{\EE_Q(h,h): h\in \R^\Omega, \norm{h}_2=1, {\rm{supp}}(h)\subseteq A\right\},
\end{align}
where $\|h\|_2=(\E_{\pi}[|h|^{2}])^{1/2} $.

We next recall Dirichlet's principle for effective resistance. For a proof 
see~\cite[Theorem~3.36]{aldous}.
\begin{theorem}\label{lem:commutetime}
        Let $Y$ be a reversible Markov chain on a finite state space $\Omega$ with transition matrix~$P$ and stationary distribution $\pi$. Then for all $a,b\in \Omega$ we have 
        \[
        \frac{1}{\estart{T_{ba}}{a}} = \inf_{f \in \R^{\Omega} :\, f (a) = 1,f(b)=0,\, \, 0 \le f \le 1} \mathcal{E}_P(f,f).
        \]
\end{theorem}

The following result is due to Doyle and Steiner~\cite{doyle2} and was also rediscovered by Gaudilli\`ere and Landim~\cite{GL}. We include its short proof for the reader's convenience.

\begin{lemma}\label{lem:symmetriz}
        Let $Y$ be a Markov chain on a finite state space $\Omega$ with transition matrix $P$ and stationary distribution $\pi$. Let $P^*$ be its time reversal and $S=(P+P^*)/2$. Then for all $a,b\in \Omega$ we have 
        \[
        \mathbb{E}^S_a[T_{ba}] \leq \mathbb{E}^P_a[T_{ba}] 
        \]
\end{lemma}

\begin{proof}[\bf Proof]
        Let $v(x)=\prstart{T_a<T_b}{x} $. Then using that $v$ is harmonic off $\{a,b\}$ (and that $\frac{1}{\E_{a}^{P}[T_{ba}] }=\pi(a)\mathbb{P}_a^{P}[T_{b}<T_a^+] $, which can be derived using Wald's equation, by considering the number of returns of $a$ before the first excursion that visits $b$) we get
\[\frac{1}{\E_{a}^{P}[T_{ba}] }=\mathcal{E}_{P}(v,v)=\mathcal{E}_{S}(v,v)\ge \frac{1}{\E_{a}^{S}[T_{ba}] }, \]
where the inequality follows from Theorem~\ref{lem:commutetime} applied to the reversible matrix $S$ using that $v(b)=0$ and $v(a)=1$.
\end{proof}

The following is similar to \cite[p.\ 12]{hermonspectral}. We give the proof for the sake of completeness.
\begin{lemma}
\label{lem:hitmaxvsstat}
There exists a positive constant $c$ so that for every reversible Markov chain on a finite state space $\Omega$ with stationary distribution~$\pi$, for all $A \subset \Omega$ we have that
\[\max_{x \in \Omega } \E_x[T_A] \le \E_{\pi}[T_A]+c \mix.  \]
\end{lemma}

\begin{proof}[\bf Proof]
We start by recalling another notion of mixing introduced by Aldous~\cite{aldous82} in the continuous time case, and later studied in discrete time by Lov{\'a}sz
and Winkler \cite{LW}. We let
\begin{align}\label{eq:deftstop}
\begin{split}
t_{\mathrm{stop}}=\max_x \inf \{\mathbb{E}_{x}[T]: T &\text{
randomised stopping time} \\ 
&\text{s.t. } \prstart{X_{T} =y}{x} = \pi(y) \text{ for all } y \}.
\end{split}
\end{align}
Note that for every $x$ there is a randomised stopping time that achieves the infimum above (see~\cite{aldous82} and~\cite{LW}). 
Let $x\in \Omega$ and let $T$ be a stopping time such that $\prstart{X_{T} \in \cdot}{x} = \pi(\cdot)$ and $\mathbb{E}_x[T] \le t_{\mathrm{stop}}$. Writing $a_+=\max\{0,a\}$ we then get
\begin{equation*}
\mathbb{E}_x[T_A] \le \mathbb{E}_x[T]+\mathbb{E}_x[(T_{A}-T)_{+}] \le t_{\mathrm{stop}}+ \mathbb{E}_{\pi}[T_{A}],  
\end{equation*}
 where we have used $\mathbb{E}_x[(T_{A}-T)_{+}] \le \mathbb{E}_x[\inf\{t \ge T  :X_{t} \in A \}-T]= \mathbb{E}_{\pi}[T_{A}]$. 
The proof is now concluded using the result of Aldous~\cite{aldous82} that there exists a positive constant $c$ so that for all reversible chains $\tstop\leq c\mix$.  
 \end{proof}

\begin{lemma}
\label{lem:PauxPsrw}
Consider the auxiliary chain $Y$ from Definition~\ref{def:aux} with parameters $\mu$ and $p$. Denote its transition matrix by $P_{\mathrm{aux},(\mu,p)}$. Then for all $x,y$ such that $\{x,y\} \in E $ we have that
\begin{align*}
P_{\mathrm{aux},(\mu,p)}(x,y) &\ge  P_{\mathrm{SRW}}(x,y)p \frac{\mu}{1+\mu} \quad \text{ and } \quad \\ \frac{P_{\mathrm{aux},(\mu,p)}(x,y)+P^*_{\mathrm{aux},(\mu,p)}(x,y)}{2} &\ge  P_{\mathrm{SRW}}(x,y)p \frac{\mu}{1+\mu}. 
\end{align*}
\end{lemma}
\begin{proof}[\bf Proof]

Start the environment at time 0 from its stationary distribution, and the walk co-ordinate of the full process from $x$. Now consider the event that the first attempt of a jump of the walk co-ordinate is to $y$,  that the edge $\{x,y\}$ is open at that time, and that before the walk attempts to make its second jump  the edge $\{x,y\}$ is refreshed. The probability of this event is at least $ P_{\mathrm{SRW}}(x,y)p \sfrac{\mu}{1+\mu} $ and at most $P_{\mathrm{aux},(\mu,p)}(x,y)$.
        
Using the first inequality, the fact that the invariant distribution for the auxiliary chain is the same as for the simple random walk on $G$ by Lemma~\ref{lem:inv} and that the simple random walk is a reversible chain, we get that for all $\{x,y\} \in E$
        \[
        \frac{P_{\mathrm{aux},(\mu,p)}(x,y)+P^*_{\mathrm{aux},(\mu,p)}(x,y)}{2}  \geq P_{\mathrm{SRW}}(x,y)p \frac{\mu}{1+\mu}
        \]
and this concludes the proof. 
\end{proof}

We are now ready to prove Theorem~\ref{thm:2}. 

\begin{proof}[\bf Proof of Theorem~\ref{thm:2}]

Let $X_0=x, \eta_0=\eta$ and suppose that the infection process $R$ starts with all edges of $G$ infected,  i.e.\ $|R_0|=|E|$. Let $\tau$ be the first regeneration time, i.e.\ the first time that the size of $R$ becomes $0$. Then for all $x,\eta$ using that $|R_t|$ evolves as a birth and death chain with $q(i,i+1)=1$ and $q(i,i-1)=\mu i$ we get that 
\begin{align}\label{eq:firstregen}
        \E[\tau] \asymp_{\mu} {\log |E|}.
\end{align}
We now claim that $\eta_\tau$ is distributed according to $\pi_p$ and is independent of $X_\tau$. Indeed, considering for every edge examined by the walk the last time before $\tau$ that this happened proves the claim, since edges that were not examined are distributed according to ${\rm{Ber}}(p)$.
        Therefore we obtain that 
        \begin{align}\label{eq:hitstat}
                \hit^{\mathrm{full},(\mu,p)} \leq  \E[\tau]+\max_{x,y\in V} \E^{(\mu,p)}_{x,\pi
                _p}\left[T_{y\times \{0,1\}^E}\right].
        \end{align}
        
Consider the auxiliary chain with parameters $\mu$ and $p$ and denote its transition matrix by $P_{\mathrm{aux},(\mu,p)}$ and let $S=(P_{\mathrm{aux},(\mu,p)}+P^*_{\mathrm{aux},(\mu,p)})/2$ be its additive symmetrisation. Let $P_{\mathrm{SRW}}$ denote the transition matrix of the simple random walk on $G$.

        Lemma~\ref{lem:PauxPsrw} immediately implies that $\EE_{S}(f,f) \geq \frac{p\mu}{1+\mu} \EE_{P_{\rm{SRW}}}(f,f)$ for all $f\in \R^V$. This together with Theorem~\ref{lem:commutetime} gives that for all $x,y\in V$
        \[
        p \sfrac{\mu}{1+\mu} \E_x^{S,(\mu,p)}[T_{yx}] \le \E_x^{\SRW}[T_{yx}].
        \]
Applying Lemma~\ref{lem:symmetriz} we obtain that 
\begin{align*}
        \E_x^{\rm{aux},(\mu,p)}[T_{yx}]\leq \E_x^{S,(\mu,p)}[T_{yx}].
\end{align*}
We next let 
\[
\til{T}_{yx} = \inf\left\{ t>T_{y\times \{0,1\}^E}: X_t = x\right\}.
\]
Let $\tau_i$ be the $i$-th regeneration time. Then by Lemma~\ref{lem:prelimforproof} the variables $\sigma_i = \tau_i-\tau_{i-1}$ are i.i.d.\ of mean~$e^{1/\mu}$. Using the natural coupling of the full process started from $(\delta_x,\pi_p)$ with the auxiliary chain started from $x$, we obtain 
\[
\til{T}_{yx} \leq \sum_{i=1}^{T_{yx}}\sigma_i,
\]
and hence using Wald's identity
\[
p \sfrac{\mu}{1+\mu} \E_x^{\pi_p,(\mu,p)}[\til{T}_{xy}] \le e^{1/\mu } \E_x^{\SRW}[T_{xy}].
\]
This together with~\eqref{eq:firstregen} and~\eqref{eq:hitstat}  now implies
\begin{align*}
\hit^{\mathrm{full},(\mu,p)} \lesssim_\mu \frac{1}{p} \hitsrw + \log |E|. 
\end{align*}
Using the fact that $\hitsrw \ge |V|-1 $ we can absorb the $\log |E|$ term inside the first term and this proves~\eqref{e:hitmain1}.

%
%
Next we prove  \eqref{e:hitmain2}. Let $x,y \in V $ and $\eta \in \{0,1\}^E$. By Lemma \ref{lem:hitmaxvsstat} \begin{equation}
\label{e:5.3}
\E_{x,\eta}^{(\mu_{i},p)}[T_{\{y\} \times \{0,1\}^E  }] \le c\mix^{\mathrm{full},(\mu_{i},p)}+ \E_{\pi_{\fu,p} }^{(\mu_{i},p)}[T_{\{y\} \times \{0,1\}^E  }] .  
\end{equation}
Recall \cite[Ch.\ 3]{aldous} that for a general reversible Markov chain on a finite state space with stationary distribution $\pi$,  for every set $A$ there exists a distribution $\mu_A$, known as the quasi-stationary distribution of $A$, such that 
\begin{equation}
\label{e:hitQS}
\E_{\pi}[T_A] \le\E_{\mu_A}[T_A]=1/\lambda(A) \le \max_x \E_x[T_A],
\end{equation}
where $\lambda(A)$ is the smallest eigenvalue of the generator of the chain killed upon exiting $A$. Applying this to $A={\{y\} \times \{0,1\}^E  }$ we have that 
\begin{align*}
\E_{\pi_{\fu,p} }^{(\mu_{1},p)}[T_{\{y\} \times \{0,1\}^E  }] &\le \frac{1}{\lambda_{(\mu_{1},p)}(\{y\} \times \{0,1\}^E)} \\
&\le \frac{\mu_2}{ \mu_1 \lambda_{(\mu_{2},p)}(\{y\} \times \{0,1\}^E)}\le \frac{\mu_2}{\mu_{1}}\hit^{\mathrm{full},(\mu_2,p)},    
\end{align*} 
where for the second inequality we used~\eqref{eq:lambdaa} together with the fact that if $\mu_1<\mu_2$, then 
\[
\mathcal{L}_1((x,\eta),(x',\eta')) \leq \tfrac{\mu_2}{\mu_1} \mathcal{L}_2((x,\eta),(x',\eta'))
\] 
with $\mathcal{L}_i$ denoting the generator of the full process when the rate at which the edges update is $\mu_i$.  This together with~\eqref{e:5.3} concludes the proof.
\end{proof}

\section{The spectral-profile}
\label{s:profiles}

We now recall a couple of results from \cite{spectral}. While some of the results below were originally stated in the case where the infinitesimal generator $\cL $ is of the form $K-I$, where $I$ is the identity matrix and $K$ is a transition matrix of a discrete-time Markov chain (possibly with non-zero diagonal entries), they hold for general $\cL$ when the state space is finite, as we can always write $\cL:=c(K-I) $ for some $c>0$ and some transition matrix $K$, possibly with positive diagonal entries. (Namely $c=\max (-\cL(x,x))$ and $K=\cL/c +I$.) (All the quantities considered below scale linearly in $c$.)

Consider an irreducible continuous-time Markov chain on a finite state space $\Omega$ with infinitesimal generator $\cL$ and stationary distribution $\pi$. Let $P_t:=e^{t \cL} $. Recall that the time $t$ transition probabilities are given by $P_t(\cdot,\cdot)$ and  that the corresponding \emph{rates} are given by $\cL(x,y):=\lim_{h \to 0} h^{-1}(P_{h}(x,y)-\1(x=y))$. Denote the infinitesimal generators of the \emph{time-reversal} and \emph{additive-symmetrization} by $\cL^* $ and $\cL^{\mathrm{s}}:=\half(\cL+\cL^*) $, respectively, where  $\cL^* $ is the dual operator of $\cL$ in $(L_2(\Omega),\langle  \cdot,\cdot \rangle_{\pi}) $, whose rates are given explicitly by $\pi(x) \cL^*(x,y):=\pi(y) \cL(y,x)$, where   $\langle  f,g \rangle_{\pi}:=\E_{\pi}[fg] $ is the inner-product on $\R^{\Omega}$ induced by $\pi$ and $\E_{\pi}[h]:=\sum_{x \in \Omega} \pi(x) h(x)$ the stationary expectation of $h \in \R^{\Omega}$. We say that $\cL$ (and the corresponding Markov chain) is \emph{reversible} if $\cL=\cL^* $. One can readily check that $\pi$ is stationary also for $\cL^*$ and that $\cL^{\mathrm{s}}$ is reversible w.r.t.\ $\pi$, as $(\cL^*)^*=\cL $.

Recall that for $\epsilon>0$ we denote the spectral profile
\begin{equation}
\label{e:la}
\Lambda(\eps)=\min \{\EE(h,h) : h \in \R^{\Omega}, \, \Var_{\pi}(h)=1,\, \pi( \mathrm{supp}(h)) \le \eps   \},
\end{equation}
 where $ \mathrm{supp}(h)=\{x \in \Omega :h(x) \neq 0 \}$ is the \emph{support} of $h$. We write $\lambda=\Lambda(1)$ for the \emph{Poincar\'e constant}. We recall that in the reversible setting $\lambda$ is equal to the spectral gap.  A related notion is
\begin{equation}
\label{e:la0}
\Lambda_{0}(\eps)=\min \{\EE(h,h): h \in \R^{\Omega}, \, \|h\|_2=1, \,  \pi( \mathrm{supp}(h)) \le \eps   \},
\end{equation}
where for $h \in \R^{\Omega}$,  $\|h\|_{\infty}=\max_{x \in \Omega} |h(x)|$ and $\|h\|_p=(\E_{\pi}[|h|^{p}])^{1/p} $ are the $L_{\infty}$ and $L_p$ ($p \in [1,\infty)$) norms respectively. 
By Cauchy-Schwarz 
\[
\|h\|_1^2=\|h \1(\mathrm{supp}(h)) \|_1^2 \le \|h\|_2^2 \pi( \mathrm{supp}(h)),
\]
and so $ 1-\pi( \mathrm{supp}(h)) \le \sfrac{\Var_{\pi}(h)}{\|h\|_2^2} \le 1$. Hence we get that
\begin{equation}
\label{e:lala0}
(1-\eps) \Lambda(\eps) \le \Lambda_{0}(\eps) \le \Lambda(\eps). 
 \end{equation}
As we now explain, $\Lambda_0$ has a probabilistic interpretation, which we shall exploit later on. 
For each $A \subsetneq \Omega $ we define $\cL_A$ to be the generator of the chain that is killed upon exiting $A$, whose rates are given explicitly by $\cL_A(u,v)=\cL(u,v)\1(u,v\in A)$, and define $\lambda(A) $ to be the smallest eigenvalue of $-\cL_A^{\mathrm{s}}=-\half(\cL_A+\cL_A^{*}) $. Under reversibility
\begin{equation}
\label{e:laexit}
\lambda(A)  =\sup \left\{c: \max_{a \in A} \E_a[\exp (c T_{A^{c}} )]< \infty \right\}=\sup \left\{c:  \E_{\pi}[\exp (c T_{A^{c}} )]< \infty \right\}. 
\end{equation}

Indeed, if $B=\sup \{c: \max_{a \in A} \E_a[\exp (c T_{A^{c}} )]< \infty \} $ and $C=\sup \{c:  \E_{\pi}[\exp (c T_{A^{c}} )]< \infty \}$, then by definition $C \ge B $, but as $\pi$ has full support we actually have that $C=B $. Now, it is classical (\cite[Ch.\ 4]{aldous} or \cite[\S 3]{Basu}) that under reversibility the law of $T_{A^{c}}$ under initial distribution $\pi$ is a mixture of  an atom at 0 of mass $\pi(A^{c}) $ and of exponential distributions whose  minimal parameter is exactly~$\lambda(A)$. From which we easily get that $\lambda(A)=C$.  

Writing $\cL=c(P-I)$ for some transition matrix $P$ and $c>0$ and using the Perron-Frobenius Theorem (to argue that the minimum can be attained by some non-negative $h$) we also have that (e.g.\ \cite[Ch.\ 3]{aldous})
\begin{equation}
\label{e:laAextremalchar}
\begin{split}
\lambda(A)  & =\min \{ \EE(h,h): h \in \R^{\Omega},\, \|h\|_2=1,\, \mathrm{supp}(h) \subseteq A    \} \\ & =\min \{ \EE(h,h): h \in \R_{+}^{\Omega},\, \|h\|_2=1,\, \mathrm{supp}(h) \subseteq A    \}, 
   \end{split}
\end{equation}
and so
\begin{equation}
\label{e:laAla0}
 \Lambda_{0}(\eps)=\min \{ \lambda(A):\pi(A) \le \eps \}. 
\end{equation}  

\subsection{Decay of $L_2$ distances  via the spectral-profile and the Poincar\'e constant}

Recall that the $L_p$ norm (w.r.t.\ $\pi$) of a signed measure $\sigma$ is defined as
\begin{equation*}
\label{eq: Lpdef}
\|\sigma \|_{p,\pi}=\|\sigma / \pi \|_p, \quad \text{where} \quad (\sigma / \pi)(x)=\sigma(x) / \pi(x).
\end{equation*}
In particular, for a distribution $\nu$ its $L_2$ distance from $\pi$ satisfies
\[\|\nu - \pi \|_{2,\pi}^2=\|\nu / \pi  - 1\|_2^{2}=\mathrm{Var}_{\pi}(\nu / \pi). \]
Let $\nu_t=\mathrm{P}_{\nu}^t:=\nu e^{\mathcal{L}t}$ and $u_t=\nu_{t} / \pi=e^{t \cL^* }(\sfrac{\nu}{\pi}) $ (where $\nu e^{\mathcal{L}t}(x)=\sum_{y}\nu(y)P_t(y,x)$, while $e^{t \cL^* }f(x)=\sum_y e^{t \cL^* }(x,y)f(y)$). For  $f \in \R^{\Omega}$, writing $f_t=e^{t \cL^* }f$ we  have $\frac{d}{dt}\mathrm{Var}_{\pi}(f_{t} )=\langle\mathcal{L}^{*}f_{t} ,f_{t} \rangle_{\pi}+\langle f_{t} ,\mathcal{L}^{*}f_{t} \rangle_{\pi}=2\langle \cL f_{t} ,f_{t} \rangle_{\pi}=-2 \EE(f_{t} ,f_{t} ) $ (cf.\ \cite[p.\ 284]{levin}) and so 
\begin{equation}
\label{e:ddtut}
\frac{d}{dt}\mathrm{Var}_{\pi}(u_{t} )=\frac{d}{dt}\mathrm{Var}_{\pi}(e^{t \cL^* }(\sfrac{\nu}{\pi}) )=-2 \EE(u_{t},u_{t}).
 \end{equation}
By~\eqref{e:la} we get $ \EE(u_{t},u_{t}) \ge  \lambda   \mathrm{Var}_{\pi}(u_{t} ) $ from which it follows that $\frac{d}{dt}\mathrm{Var}_{\pi}(u_{t} ) \le -2\lambda  \mathrm{Var}_{\pi}(u_{t} ) $, and so by Gr\"onwall's lemma
\begin{equation}
\label{e:Poincare}
\|\nu_{t} - \pi \|_{2,\pi}^2 \le \|\nu - \pi \|_{2,\pi}^2 \exp(-2\lambda t).  
\end{equation}
This is the well-known Poincar\'e inequality. The $\eps$ $L_p$-\emph{mixing time} is defined as
\[t_{\mathrm{mix}}^{(p)}(\eps)=\inf\{t:\max_{x}\|P^t(x,\cdot) - \pi\|_{p,\pi} \le \eps \}. \]
 It is standard (\textit{e.g.}\ \cite{spectral} or \cite[Prop.\ 4.15]{levin}) that for reversible Markov chains, for all $x \in \Omega $ and $t$ we have  \begin{equation}
\label{e:maxdiag}
\max_{y,z \in \Omega}|\sfrac{P_t(y,z)}{\pi(z)}-1|=\max_{y}\sfrac{P_t(y,y)}{\pi(y)}-1  \quad \text{and} \quad \|P_{x}^t - \pi \|_{2,\pi}^2=\sfrac{P_{2t}(x,x)}{\pi(x)}-1.
\end{equation}  Thus  $t_{\mathrm{mix}}^{(\infty)}(\eps^2)=2 t_{\mathrm{mix}}^{(2)}(\eps) $ for all $\eps \le ( \max_x \frac{1-\pi(x)}{\pi(x)})^{1/2}$. 

Theorem 1.1 in \cite{spectral} asserts that
\begin{equation}
\label{e:spb1}
\forall \, \eps \in (0,1/\pi_* ] , \quad t_{\mathrm{mix}}^{(\infty)}(\eps) \le t_{\mathrm{spectral-profile}}(\eps).
\end{equation}
We shall use a variant of this where we want to bound the $L_{2}$ mixing-time starting from some initial distribution $\nu$ for which we have a decent upper bound on $\|\nu-\pi\|_{2,\pi}$.
\begin{proposition}[\cite{spectral} Lemma 2.1]
\label{p:spectral2} For any (non-zero)  $u \in \R_{+}^{\Omega}$ we have that
\[
\frac{\EE(u,u)}{\Var_{\pi}(u)} \ge \half \Lambda \left( 4\|u\|_1^2/ \Var_{\pi}(u) \right).
\]
\end{proposition}
We recall the proof from \cite{spectral} for the reader's convenience.
\begin{proof}[\bf Proof]
Let $M=\Var_{\pi}\left(\frac{u}{\sqrt{4\|u\|_{1}}}\right)$ and $B=\{x:\, u(x) \ge M \}$. The set $B$ is non-empty by H\"older's inequality.  Considering $\hat u:=(u-M)1_{B}=(u-M)_+ $ we now have $\EE(u,u) \ge \EE(\hat u, \hat u) $. Also,  \[  \Var_{\pi}(\hat u)  \ge \E_{\pi}[(u-M)_{+}^{2}]-\left(\E_{\pi}[u]\right)^{2} \ge \|u\|_2^2-2M\|u\|_1-\|u\|_1^2  \]
\[= \Var_{\pi}(u)-2M\|u\|_1 =\half  \Var_{\pi}(u). \]  Finally, $2\sfrac{\EE(u,u)}{\Var_{\pi}(u)} \ge \sfrac{\EE(\hat u, \hat u)}{\Var_{\pi}(\hat u)} \ge \Lambda \left( \pi(\mathrm{supp}(\hat u)) \right) \ge \Lambda (\sfrac{ 4\|u\|_1^2}{ \Var_{\pi}(u)}) $, since by Markov's inequality $\pi(\mathrm{supp}(\hat u)) \le \sfrac{\|u\|_1}{M}  =\sfrac{ 4\|u\|_1^2}{ \Var_{\pi}(u)}  $. 
\end{proof}
Recall that $u_t:=\nu_t /\pi $ and $\nu_t:=\mathrm{P}_{\nu}^t$. Using Proposition \ref{p:spectral2} and \eqref{e:ddtut} it is not difficult to verify the assertion of the following proposition,  which  refines \eqref{e:Poincare} and \eqref{e:spb1}.
\begin{proposition}[\cite{spectral} Theorem 1.1]
\label{p:spectral3} 

For any initial distribution $\nu $ we have that
\begin{equation}
\label{e:spb2}
\|\nu_{t} - \pi \|_{2,\pi}^2 \le M , \quad \text{if} \quad t \ge \int_{4/\|\nu - \pi \|_{2,\pi}^2}^{4/M} \frac{d \delta}{\delta \Lambda(\delta) }. 
  \end{equation}
  In particular, for all $0<\beta <1$ we have that 
\begin{equation}
\label{e:spb3}
\|\nu_{t} - \pi \|_{2,\pi}^2 \le \beta  \|\nu - \pi \|_{2,\pi}^2  , \quad \text{if} \quad t \ge \frac{\log (1/\beta )}{\Lambda \left(4/(\beta  \|\nu - \pi \|_{2,\pi}^2)\right)}  .
 \end{equation}
\end{proposition}
We shall sometimes use the following simple variant of \eqref{e:spb3}.

\begin{lemma}
If $\|\nu - \pi \|_{2,\pi}^2 \le M $ then
for all $0<\beta <1$ we have that 
\begin{equation}
\label{e:spb4}
\|\nu_{t} - \pi \|_{2,\pi}^2 \le \beta  M  , \quad \text{if} \quad t \ge \frac{\log (1/\beta )}{\Lambda \left(4/(\beta M)\right)}  .
 \end{equation}
 \end{lemma}
\begin{proof}[\bf Proof]
If $\|\nu - \pi \|_{2,\pi}^2 \le \beta M$ this follows from the fact that $P_t$ is a contraction in $L_2$, i.e.\ $\|\nu_{t} - \pi \|_{2,\pi}^2 \le \|\nu - \pi \|_{2,\pi}^2 \le \beta M $. If $\|\nu - \pi \|_{2,\pi}^2 \in (\beta M,M ] $ then by \eqref{e:spb3} for $\beta'=\sfrac{\beta M}{\|\nu - \pi \|_{2,\pi}^2 } \ge \beta $, $s=\frac{\log (1/\beta' )}{\Lambda \left(4/(\beta'  \|\nu - \pi \|_{2,\pi}^2)\right)} =\frac{\log (1/\beta' )}{\Lambda \left(4/(\beta  M)\right)}\le \frac{\log (1/\beta )}{\Lambda \left(4/(\beta  M)\right)} $ we have  $\|\nu_{t} - \pi \|_{2,\pi}^2 \le \beta' \|\nu - \pi \|_{2,\pi}^2=\beta M$.  
\end{proof}

 \subsection{A lower bound on $L_2$ distances in terms of small sets probabilities}
Let $\PP(\Omega)$ be the collection of all distributions on $ \Omega $. Let $A  \subsetneq \Omega $ and $\delta \in (0,1) $.  Let \[\PP_{A,\delta}:=\{\nu \in \PP(\Omega): \nu(A) \ge \pi(A)+\delta \pi (A^c) \}.\] Note that \[\nu_{A,\delta}:= \delta \pi_A+(1-\delta)\pi \in \PP_{A,\delta}, \]
where $\pi_A(a):=\frac{1_{\{a \in A\}\pi(a) }}{\pi(A)}$ is $\pi$ conditioned on $A$. Moreover, $\min \{\delta':\nu_{A,\delta'} \in \PP_{A,\delta}   \}=\delta $. It is thus intuitive that for a convex distance function between distributions, $\nu_{A,\delta} $ is the closest distribution to $\pi$ in $\PP_{A,\delta}$.  
\begin{proposition}[\cite{L2} Proposition 4.1]
\label{prop: Lagrange}
 Let $A  \subsetneq V $. Denote $\nu_{A,\delta}= \delta \pi_A+(1-\delta)\pi$.  Then
\begin{equation}
\label{eq: Lagrange}
\forall \, \delta \in (0,1)  \quad  \min_{\nu \in \PP_{A,\delta}}\| \nu-\pi\|_{2,\pi}^2=\|\nu_{A,\delta}-\pi\|_{2,\pi}^2=\delta^{2} \pi(A^{c})/\pi(A).
\end{equation}
\end{proposition}
The proof is an exercise in Lagrange multipliers (see \cite{L2} Proposition 4.1).
\subsection{A mixing time bound for the $p$-tilted hypercube}
We call the hypercube $\{0,1\}^d$ equipped with the product measure $\pi_p$ the $p$-\emph{tilted hypercube}. The natural dynamics associated with it is the one at which each co-ordinate is updated at rate $\mu$ and takes the value $1$ with probability $p$ and $0$ with probability $1-p$. This is precisely the evolution of the environment $\eta$.
\begin{lemma}
\label{lem:ptiltedhypermixing}
Denote by $P_t$ the time $t$ transition kernel of the $p$-tilted hypercube with update rate $\mu$ for each co-ordinate. Let $\alpha :=p \wedge (1-p)$.  Let $t(\delta)=t_{p,\mu}(\delta):=\frac{1}{\mu}\log(d \sfrac{1-\alpha}{\alpha\log(1+\delta)})$. Then 
\begin{equation}
\label{e:ptiltedmixing}
\max_{x,y \in \{0,1\}^d  }\left|\frac{P_{t(\delta)}(x,y)}{\pi_{p}(y)}-1 \right| \le \delta.
\end{equation}
\end{lemma}
\begin{proof}
By scaling we can assume $\mu=1$. Denote the transition kernel of a single co-ordinate by $Q_t$. Let $\nu_{p}(1):=p=:1-\nu_{p}(0)$. Then $Q_t(a,a)=e^{-t}+(1-e^{-t})\nu(a) $ for $a \in \{0,1\}$.   Now
\[\frac{P_t(x,x)}{\pi_{p}(x)}=\prod_{i \in [d] }\frac{Q_t(x_{i},x_{i})}{\nu_{p}(x_{i})}  \le \prod_{i \in [d] }\left(1+ \frac{1-\alpha}{\alpha e^{t}} \right) \le \exp \left(d \frac{1-\alpha}{\alpha e^{t}} \right) \]
(using $1+x \le e^x$). The proof is concluded using \eqref{e:maxdiag}  by substituting $t=t(\delta)$ above. 
\end{proof}

\section{Proof of Theorem \ref{thm:1}}
\label{s:proof1}

\begin{definition}
\label{def:EF}
\rm{
Let $A \subset V \times \{0,1\}^E $. For every $a \in V$ we define 
\[
\mathrm{Env}(a,A)=\{\eta \in \{0,1\}^E :(a,\eta) \in A  \}.
\]
For $\alpha\in [0,1]$ we say that $a \in V $ is $(A,\alpha)$-\emph{environmentally friendly} if $\pi_p(\mathrm{Env}(a,A)) \ge \alpha $. We denote the collection of  $(A,\alpha)$-environmentally friendly vertices by   
$A(\alpha)$.
}       
\end{definition}

From now on we fix a set $A\subset V \times \{0,1\}^E$ with $ \pi_{\fu,p} (A) \ge \half $ and set $B=A(\sfrac{1}{4} )$. For every $b \in B$ let $ \mathrm{\widehat {Env}}(b) \subseteq \mathrm{Env}(b,A) $ be some set of environments such that $\pi_p(\mathrm{\widehat {Env}}(b)) \in [\sfrac{1}{4},\sfrac{1}{2}]$. Suppose that $\eta_0\sim \pi_p$. Recall from Definition~\ref{def:aux} that $(\tau_i)$ is the sequence of regeneration times with $\tau_0=0$. Let~$\widehat \tau_i:=\tau_{\inf \{\ell:\sum_{j=1}^{\ell}1_{\{X_{\tau_{j}} \in B \}}=i \}}$ be the $i$-th regeneration time at which the walk co-ordinate is in $B $. We now take a subsequence defined as follows: $\sigma_1=\widehat \tau_1$ and inductively
\[
\sigma_{i+1}=\inf \{ \widehat \tau_j : \widehat \tau_j  \ge \sigma_i + \kappa \}, 
\]
where $\kappa$ is a constant to be determined later. Finally we let 
\[
T=\inf \left\{j:\eta_{\sigma_j} \in \mathrm{\widehat {Env}}( X_{\sigma_j}) \right\}.
\]

\begin{lemma}\label{lem:statindep}
        Let $\eta_0\sim \pi_p$. Then $\eta_{\sigma_1}$ is independent of $X_{\sigma_1}$ and distributed according to $\pi_p$.
\end{lemma}

\begin{proof}[\bf Proof]
        If an edge has not been examined by the walk during $[0,\sigma_1]$, then at time $\sigma_1$ it is distributed according to $\mathrm{Ber}(p)$. For the edges that were examined by the walk, considering the last time before time~$\sigma_1$ that this happened we get that at time $\sigma_1$ they are also distributed according to $\mathrm{Ber}(p)$ independently over different edges and of the location of the walk. 
\end{proof}

\begin{lemma}
\label{lem:keyhyper}
If $\kappa$ in the definition of $(\sigma_i)$ is taken sufficiently large, then
 \begin{eqnarray*}
\forall \, i \in \N, \qquad \prcond{T=i}{T \ge i}{}& \ge & \sfrac{1}{8}.
 \end{eqnarray*}
  \end{lemma}

Let $Y$ be the auxiliary chain with parameters $\mu=1$ and $p\in (0,1)$. We take its continuous time version, i.e.\ we consider the continuous time chain which stays at every vertex for an exponential time of parameter $1$ and then makes a jump according to $P_{\mathrm{aux}}$. 
We write $(P^{\mathrm{aux, cts}}_t)_{t\in \R_+}$ for its transition semigroup, i.e.\ $P^{\mathrm{aux,cts}}_s=e^{s(P_{\mathrm{aux}}-I)} $,  and define for $\delta>0$
\begin{eqnarray}
\label{e:r}
r(\delta) & = & \inf \left\{s\in \R_+: \max_{x,y \in V } P^{\mathrm{aux,cts}}_{t}(x,y) \le 1-\delta \text{ for all }t \ge s \right\}.
\end{eqnarray}

\begin{lemma}
\label{lem:r}
There exists $\delta_0\in (0,1)$ such that for all $p\in (0,1)$ we have  
\begin{eqnarray}
\label{e:r2}
r(\delta_0) & \leq & \frac{1}{p}.
\end{eqnarray}
\end{lemma}

\begin{proof}[\bf Proof]

We write $J_t=P^{\mathrm{aux,cts}}_t  $. We first note that $\max_{x,y \in V } J_{t}(x,y) $ is non-increasing in $t$ and so $r(\delta)=\inf \{s \in \R_+ : \max_{x,y \in V } J_{s}(x,y)\le 1-\delta  \} $. Now, let $s=1/p $ and let $T_1$ be the first time the continuous time auxiliary chain jumps out of $x$. Then $T_1$ is exponential of parameter $1-P_{\mathrm{aux}}(x,x)\asymp p$ by Lemma~\ref{lem:holdingprobforaux}. We then get 
\[
J_s(x,x) \geq \pr{T_1>s} = e^{-s(1-P_{\mathrm{aux}}(x,x))} \geq  c_1.
\]
Let $T_2$ be an exponential random variable independent of $T_1$ of parameter $\max_y(1-P_{\mathrm{aux}}(y,y))\asymp p$. We then get 
\[
1-J_{s}(x,x) \geq \prstart{T_1<s, T_2>s}{x} \geq c_2
\]
for a positive constant $c_2$. Noting that $ J_{s}(x,y) \le1- J_{s}(x,x)   $ for $x \neq y$ concludes the proof.   
\end{proof}

Below we write $\Lambda_{{\rm{full}},(\mu,p)}$, $\Lambda_{{\rm{aux}}, (\mu,p)}$ and $\Lambda_{\SRW}$ for the spectral profile of the full chain, the auxiliary chain with parameters $(\mu,p)$ and the simple random walk respectively. Note that $\Lambda_{{\rm{aux}}, (\mu,p)}$ is defined with respect to the generator $P_{{\rm{aux}}, (\mu,p)}-I$.
We write $\cL_{\mu,p}$ for the generator of the full process with parameters $(\mu,p)$. Finally, for $A \subset V \times \{0,1 \}^{E} $ we write  $\lambda_{(\mu,p)}(A)$ for the smallest eigenvalue of the restriction of $-\cL_{\mu,p}$ to $A$, as in the paragraph preceding~\eqref{e:laexit}.

\begin{proposition}
\label{prop:LDaux}
There exist positive constants $M$ and $L$ so that the following holds. 
Let $A \subset V \times \{0,1\} $ be such that $\pi_{\fu,p}(A^{c}) \le 1/2$, let $B=A(\sfrac{1}{4}) \subseteq V $ be as above, let $\delta_0$ be as in Lemma~\ref{lem:r} and $(Y_{s})_{s\geq 0}$ be the continuous time chain with generator $P_{{\rm{aux}}, (1,p)}-I$ with $p\in (0,1)$. There exists a sequence of stopping times $T_1< T_2< \cdots $ (w.r.t.\ the chain $(Y_s)_{s\geq 0}$) such that for all $i$ we have $Y_{T_i}\in B$ and  $T_i/\left\lceil \frac{M}{\Lambda_{\mathrm{aux},(1,p)}(M \pi_{\fu,p}(A^{c})  )}+r(\delta_0)    \right\rceil $ is stochastically dominated by the law of $\sum_{j=1}^i Z_j $, where $Z_1,Z_2,\ldots$ are i.i.d.\ Geometric random variables with mean at most~$L$, where~$r(\delta_0)$ is as in~\eqref{e:r}. 
\end{proposition}

We defer the proofs of Lemma~\ref{lem:keyhyper} and Proposition~\ref{prop:LDaux} until after the proof of Theorem~\ref{thm:1}.
 
\begin{proof}[\bf Proof of Theorem \ref{thm:1}]

To simplify notation we write $r=r(\delta_0)$.

The proof is mostly a formal exercise involving translating the assertion of Proposition \ref{prop:LDaux} concerning the  rate of exponential decay of the tail of $T_i/\left(i\left\lceil \frac{M}{\Lambda_{\mathrm{aux},(1,p)}(M \pi_{\fu,p}(A^{c})  )}+r    \right\rceil \right) $ into one about  the rate of exponential decay of the tail of $\sigma_i/\left(i\left\lceil \frac{M}{\Lambda_{\mathrm{aux},(1,p)}(M \pi_{\fu,p}(A^{c})  )}+r    \right\rceil \right) $. This is straightforward in light of the fact that the spacings between the regeneration times are i.i.d.\ with an exponentially decaying tail, and that each such spacing is at least $\kappa$ w.p.\ at least $e^{-\kappa}$.  We now give the formal details.

Let $t=\frac{1}{2\mu} \log \left(|E|(1-\alpha)/(\alpha \log 2)\right)$, where $\alpha=p\wedge (1-p)$. Then it is not hard to see that 
\begin{align}\label{eq:normbound}
\max_{x,\eta} \|\P_{(x,\eta)}^t- \pi_{\fu,p} \|_{2,\pi_{\fu,p}}^2 \le \frac{2}{\min_{v \in V}\pi(v)}. 
\end{align}
Indeed, let $P_t=P_t^{(\mu,p)}$ and $Q_t=Q_t^{(\mu,p)}$ be the transition kernels for time $t$ of the full process and of the environment, respectively. The latter is simply a continuous-time SRW on the $p$-tilted hypercube, and its $1$-$L_2$ mixing time (i.e.\ $t_{\mathrm{mix}}^{(2)}(\varepsilon)$ with $\varepsilon=1$) is at most $t$ by Lemma~\ref{lem:ptiltedhypermixing}. 
 Thus 
 $\sfrac{P_{2t}((x,\eta), (x,\eta))}{\pi_{\fu,p}(x,\eta)} \le   \sfrac{Q_{2t}(\eta, \eta)}{\pi_{p}(\eta) \min_{v \in V}\pi(v)} \le 2/\min_{v \in V}\pi(v)$ which shows~\eqref{eq:normbound}.

Using~\eqref{eq:normbound} and~\eqref{e:spb2} we get
\begin{equation}
\label{e:burnin}
\mixfmupi (\eps^{2}) \le  2t+2 \int_{2\min_{v \in V}\pi(v)}^{4/\eps} \frac{d \delta}{\delta \Lambda_{\fu,(\mu,p)}(\delta) }. 
\end{equation}

Let $\mu_1<\mu_2$. Then $\cL_{\mu_2,p}((x,\eta),(y,\eta')) \le \frac{\mu_2}{\mu_{1}}\cL_{\mu_1,p}((x,\eta),(y,\eta'))$ for all $(x,\eta),(y,\eta') \in V \times \{0,1\}^E $. Using also~\eqref{e:laAextremalchar}, \eqref{e:la} and noting that both $\cL_{\mu_1,p} $ and $\cL_{\mu_2,p}$ are reversible w.r.t.\ $ \pi_{\fu,p}$ we obtain
\begin{equation}
 \label{e:lambdaA123}
 \Lambda_{\fu,(\mu_{2},p)}(\delta) \; \le \; \frac{\mu_2}{\mu_{1}} \Lambda_{\fu,(\mu_1,p)}(\delta).
 \end{equation}

By combining~\eqref{e:burnin} and~\eqref{e:lambdaA123} together with \eqref{e:lala0} and \eqref{e:laAla0}  we see that in order to conclude the proof of Theorem \ref{thm:1} it suffices to consider $\mu=1$ and prove that for a positive constant $M$ we have that for all $A \subset V \times \{0,1\}^E $ with $\pi_{\fu,p}(A^c) \le 1/2 $ 
\begin{equation}
\label{e:wtp}
\frac{p }{M\lambda_{(1,p)}(A^{c})} \; \le \; \frac{1 }{\Lambda_{\SRW}(M\pi_{\fu,p}(A^{c})  )}
\end{equation}
(the l.h.s.\ is defined w.r.t.\ the full process).
 We write $\Lambda^{\mathrm{aux},(\mu,p)}$ for the spectral profile of the auxiliary chain with parameters $\mu$ and~$p$.

By Lemma~\ref{lem:PauxPsrw} we have that for all $\delta>0$
\[ \frac{p}{\Lambda^{\mathrm{aux},(1,p)}( \delta )}\; \le \; \frac{2}{\Lambda_{\SRW}( \delta )}. \]
 Hence to conclude the proof of Theorem \ref{thm:1} it suffices to show that for $A$ as above 
\begin{equation}
\label{e:wtp2}
\frac{1 }{M\lambda_{(1,p)}(A^{c})} \; \le \; \frac{1 }{\Lambda^{\mathrm{aux},(1,p)}(M\pi_{\fu,p}(A^{c})  )}.
\end{equation}

Throughout the remainder of the section we fix $\mu=1$. Our strategy for proving \eqref{e:wtp2} is to find a set $B \subseteq V $ with $\pi(B^{c}) \lesssim \pi_{\fu,p}(A^{c})  $ and such that the asymptotic rate of decay of the tail of $T_A$ can be controlled via the time spent in $B$ by the auxiliary chain. Roughly speaking, we want to have a set $B$ such that every visit of the auxiliary chain to $B$ is a visit of the full process to $A$ with some probability bounded away from $0$. Here we are using the fact that the auxiliary chain can be coupled with the full process by viewing it along regeneration times.

We next claim that it suffices to prove that for $M$ and $r$ as in Proposition \ref{prop:LDaux} there exists a positive constant $c$ such that 
\begin{align}\label{eq:goalinproof}
        \E_{\pi_{\fu,p}}[e^{32\alpha \sigma_i }   ]  \le  e^{i} \quad \text{ for some }\, \alpha \,\text{ s.t. } 
        \frac{c}{\alpha} \leq \frac{1}{\Lambda_{\mathrm{aux},(1,p)}(M \pi_{\fu,p}(A^{c})  )}+r.
\end{align}
Indeed, for $(\sigma_i)$ and $T$ as defined earlier
\[
\prstart{T_{A}>t}{\pi_{\fu,p}} \le \prstart{\sigma_{\lceil 16 \alpha t \rceil} >t}{\pi_{\fu,p}}+\prstart{T>\lceil 16 \alpha t \rceil}{\pi_{\fu,p}}.
\]
Now by Lemma~\ref{lem:keyhyper} we have $\prstart{T>\lceil 16 \alpha t \rceil}{\pi_{\fu,p}}\le (7/8)^{\lceil 16 \alpha t \rceil} \le e^{-2\alpha t } $ and $\prstart{\sigma_{\lceil 16 \alpha t \rceil} >t}{\pi_{\fu,p}}\le \E_{\pi_{\fu,p}}[e^{32\alpha \sigma_{\lceil 16 \alpha t \rceil}  }]e^{-32\alpha t} $, which by our assumption is at most $e^{-32\alpha t}e^{\lceil 16 \alpha t \rceil}\le e^{1-16\alpha t}$. Therefore, this would then imply that $\estart{e^{\alpha T_A}}{{\pi_{\fu,p}}}<\infty$, and hence $\lambda_{(1,p)}(A^c)\geq \alpha$. 

So we now turn to prove~\eqref{eq:goalinproof}. Recall that $\tau_i$ is the $i$-th regeneration time as in Definition~\ref{def:aux}. 
We now consider the discrete time auxiliary chain $Y$. Let $\til{Y}$ be the continuous time version of $Y$, i.e.\ we let $N$ be an independent Poisson process of rate $1$ and set $\til{Y}_t=Y_{N(t)}$.  We set 
        \[
        S=8\left\lceil \frac{M}{\Lambda_{\mathrm{aux},(1,p)}(M \pi_{\fu,p}(A^{c})  )}+r    \right\rceil.  
        \]
        Hence, it suffices to show that by setting $L$ to be sufficiently large there exists a positive constant~$c$ such that for all $x\geq 1$ and all $i$
        \[ 
        \pr{\sigma_i > xL^{3}  Si} \leq c  e^{-10xi}. 
        \]
        Let $(T_i)$ be the stopping times from Proposition~\ref{prop:LDaux}. Recall that $N(t) \sim \mathrm{Poisson}(t)$ is the number of jumps the continuous-time version of the auxiliary chain makes by time $t$. Hence $N(T_i)$ is the number of jumps it makes by the stopping time $T_i$. We have generated the full process, the auxiliary chain $(Y_j)$, and its continuous-time version $\til{Y}_t:=Y_{N(t)}$ on the same probability space (the auxiliary chain is generated by viewing the walk co-ordinate of the full process at regenaration times, and the continuous-time version of the auxiliary chain is generated from the auxiliary chain by using an independent rate one Poisson process $(N(t):t \ge 0)$). Hence we can consider $\rho(i):=\tau_{N(T_i)}$, which is the time at which the $N(T_i)$-th regeneration time of the full process occurs. The walk co-ordinate of the full process at time $\rho(i)$ is $\til{Y}_{T_i}$.   
        For all $i$ we set $Z_i$ to be the time between $\rho(i)$ and the first time after $\rho(i)$ that the walk $X$ examines an edge. (Note $Z_i$ is an exponential variable of parameter $1$.) We also define 
        \[
        \xi_j = \1(N(T_j)-N(T_{j-1})\geq 1, Z_j \geq \kappa) \quad \text{ and } \quad  J_\ell = \sum_{j=1}^{\ell} \xi_j.
        \]
        Note that the variables $(\xi_i)$ are i.i.d.\ and since $T_j-T_{j-1}\geq 1$ for all $j$ and the two events appearing in the definition of $\xi_j$ are independent, we get that $\pr{\xi_j=1}\geq (1-e^{-1})e^{-\kappa}$.
  Using the definitions above we then have the following inclusions for all $x$ and $i$
        \begin{align*}
                \{\sigma_i\geq x L^3 Si\} &\subseteq \{ \rho(\lceil xLi/4\rceil)\geq x L^3 S i\}\cup\{J_{\lceil xLi/4\rceil}< i\} \\
                \{ \rho(\lceil xLi/4\rceil)\geq x L^3 S i\}&\subseteq \{N(T_{\lfloor xLi/4\rfloor})\geq xL^2Si\} \cup \{\tau_{\lceil xL^2Si\rceil} \geq x L^3Si\}\\
                 \{N(T_{\lfloor xLi/4\rfloor})\geq xL^2Si\}&\subseteq \{        T_{\lceil xLi/4\rceil} \geq xL^2Si/100\} \cup \{ N(xL^2Si/100)\geq xL^2Si\}.
        \end{align*}
(For instance, the first inclusion follows by noting that if $J(j) \ge \ell$ and $\rho(j) \le t$ then we must have that $\sigma_{\ell} \le t$. Indeed, if $J(j) \ge \ell$ then among $T_{1},\ldots,T_{j}$ there are at least $\ell$ that contribute $+1$ towards increasing the index of $\sigma$ - i.e. towards $\inf\{k:\sigma_k \ge \rho(j) \} $. If moreover $\rho(j) \le t$, then $T_{j}$ is generated in the full process at some time which is smaller or equal to $t$.)

        The proof is now concluded by taking $L$ sufficiently large and using Proposition~\ref{prop:LDaux} for the tails of the stopping times~$T_i$, the fact that $J$ is the sum of i.i.d.\ indicators with probability bounded away from $0$, large deviations for Poisson random variables and Lemma~\ref{lem:prelimforproof} for the tails of $(\tau_i-\tau_{i-1})$.
        
        Using that $t_{\rm{spectral-profile}}^{\rm{SRW}}(\epsilon) \gtrsim (\log|V|/\epsilon)$ for all $\epsilon\in (0,1)$ shows that we can absorb the logarithmic terms (one of which is coming from the term $\int_{2\min_{v \in V}\pi(v)}^{4/\epsilon}\frac{r}{\delta}d \delta$, using $\min_{v}\pi(v) \ge |V|^{-2}$, while the other one from \eqref{e:burnin}) and this completes the proof.
\end{proof}
 
\begin{remark}
\label{r:TV}
        \rm{
        We now explain how to get rid of the $\log $ term in the statement of Theorem~\ref{thm:1} when considering total variation mixing. Let $t=\frac{2\log n}{\mu}$ and $A$ be the event that all edges of $G$ have been updated by time $t$. Define $\nu_1=\mathbb{P}^{t}_{(x,\eta)}(\cdot|A)$ and $\nu_2=\mathbb{P}^{t}_{(x,\eta)}(\cdot|A^c)$.  Then we have
                \[
        \mathbb{P}^{t+s}_{(x,\eta)} = \pr{A} \mathbb{P}^s_{\nu_1} +\pr{A^c}\mathbb{P}^s_{\nu_2},
        \]
        and hence, by convexity and Jensen's inequality we obtain
        \begin{align*}
                \tv{\mathbb{P}^{t+s}_{(x,\eta)} - \pi_{\mathrm{full},p}} \leq \pr{A^c} + \|\mathbb{P}^s_{\nu_1}-\pi_{\mathrm{full},p}\|_{2,\pi_{\mathrm{full},p}}.
        \end{align*}
        Similarly to~\eqref{eq:normbound} we have that 
        \[
        \|{\nu_1}-\pi_{\mathrm{full},p}\|_{2,\pi_{\mathrm{full},p}}^2 \leq \frac{1}{\min_v \pi(v)}.
        \]
        The rest of the proof is identical to the proof of Theorem~\ref{thm:1}.}
\end{remark} 
 
%
%
%
\begin{lemma}
\label{lem:23}
Let $\pi$ be a distribution of full support on a finite set $\Omega$. Let  $J$ be an $\Omega$-valued random variable and $A$ an event. Then
\begin{equation}
\label{e:hatnuL22}
\|\P( J \in \cdot \mid A )- \pi \|_{2,\pi}^2 \; \le \; \frac{\| \P( J \in \cdot  )- \pi \|_{2,\pi}^2 +1 }{\P(A)^2}-1.
\end{equation}
In particular, if $\nu$ is some distribution on $\Omega$ and $\widehat \nu$ is $\nu$ conditioned on $A \subseteq \Omega $, then
\begin{equation}
\label{e:hatnuL2}
\|\widehat \nu - \pi \|_{2,\pi}^2 \; \le \; \frac{\| \nu - \pi \|_{2,\pi}^2 +1 }{\nu(A)^2}-1.
\end{equation}
\end{lemma}
\begin{proof}[\bf Proof]
We have that $\|  \P( J \in \cdot  )- \pi \|_{2,\pi}^2  +1= \sum_{x  } \pi(x)\left(\sfrac{  \P( J =x  )}{\pi(x)}  \right)^2\ge \sum_{x } \pi(x)\left(\sfrac{ \P(\{ J =x \} \cap A )}{\pi(x)}  \right)^2 $. By the same reasoning  $\|\P(J \in \cdot \mid A ) - \pi \|_{2,\pi}^2  +1=   \frac{1}{\P(A)^2}\sum_{x \in A } \pi(x)\left(\sfrac{ \P(\{ J =x \} \cap A )}{\pi(x)}  \right)^2 $.   
\end{proof}

\begin{proof}[\bf Proof of Lemma~\ref{lem:keyhyper}]
We need to show that 
\[
\prcond{\eta_{\sigma_i}\in \mathrm{\widehat {Env}}( X_{\sigma_i})}{T\geq i}{} \geq \frac{1}{8}.
\]
Let $J_i$ be the set of edges examined by the walk during the time interval $[\sigma_{i-1},\sigma_i]$ and let $E_i=E\setminus J_i$. 
Crucially, given $(J_i,E_i)$, $\eta_{\sigma_{i-1}}$ and the whole history $\sigma(X_t:t\leq \sigma_i)$ we have that the law of $\eta_{\sigma_i} $ can be described as follows:
\begin{itemize}
\item[(1)] The different co-ordinates of  $\eta_{\sigma_i} $ are independent;
\item[(2)] For $e \in J_i $ we have that $\eta_{\sigma_i}(e) \sim \mathrm{Bernoulli}(p) $;
\item[(3)] For $e \in E_i $ we have that $\eta_{\sigma_i}(e) \sim \mathrm{Bernoulli}(p) $ with probability $1-e^{-(\sigma_i-\sigma_{i-1})} $ and otherwise $\eta_{\sigma_i}(e) =\eta_{\sigma_{i-1}}(e)$. In other words, the restriction of the environment to $E_i$ evolves during $[\sigma_{i-1},\sigma_i]$ by updating each edge at rate $1$ to be either open w.p.\ $p$ or closed w.p.\ $1-p$.
\end{itemize} 
For any sequence $(x_i)$ with $x_i\in B$ for all $i$ and all sequences $(S_i)$ of subsets of $E$, define $A_i=\{X_{\sigma_1}=x_1,\ldots, X_{\sigma_i}=x_i, E_1=S_1,\ldots, E_i=S_i\}$. For every subset of edges $S$ we denote by $\pi_p^S$ the Bernoulli$(p)$ product measure on $S$. 
By properties (1) and (2) above we then have
\begin{align*}\norm{\prcond{\eta_{\sigma_i}\in \cdot}{T\geq i, A_i}{} - \pi_p}_{2,\pi_p} = \norm{\prcond{\eta_{\sigma_i}\vert_{S_i}\in \cdot}{T\geq i, A_i}{} - \pi_p^{S_i}}_{2,\pi_p^{S_i}},
\end{align*}
where $\eta\vert_{S}$ denotes the restriction of $\eta$ on the set of edges $S$. 
For every $\eta \in \{0,1\}^{S_i}$ we now let 
\begin{align*}
        \nu_i^{\eta}(\cdot)= \prcond{\sigma_{i}-\sigma_{i-1}\in \cdot}{\eta_{\sigma_{i-1}}\vert_{S_i}=\eta, T\geq i, A_i}{}.
\end{align*}
We note that under the conditioning above, the distribution of $\sigma_i-\sigma_{i-1}$ does not depend on $\{\eta_{\sigma_{i-1}}\vert_{S_i}=\eta\}$, since $S_i$ is the set of edges that the walk does not examine during $[\sigma_{i-1},\sigma_i]$. Therefore, we get for all $\eta$
\begin{align*}
\nu_i^{\eta}(\cdot) = \prcond{\sigma_{i}-\sigma_{i-1}\in \cdot}{T\geq i, A_i}{} =:\nu_i(\cdot).    
\end{align*}
Using again that $S_i$ is the set of edges that the walk does not examine during $[\sigma_{i-1},\sigma_i]$ we get
\begin{align}\label{eq:muii-1}
\mu_i(\eta) :=  \prcond{\eta_{\sigma_{i-1}}\vert_{S_i}=\eta}{T\geq i, A_i}{} = \prcond{\eta_{\sigma_{i-1}}\vert_{S_i}=\eta}{T\geq i, A_{i-1}}{}.
\end{align}
Conditional on $\{T\geq i\}\cap A_i$, the distribution of $\eta_{\sigma_{i}}\vert_{S_i}$ is that of a $p$-tilted random walk on the hypercube $
\{0,1\}^{S_i}$ started from $\eta_{\sigma_{i-1}}\vert_{S_i}$ and run for time $\sigma_i-\sigma_{i-1}$. Let $\til{\eta}$ be a continuous time $p$-tilted random walk on $
\{0,1\}^{S_i}$. 
Hence, putting all things together (and recalling that $\sigma_i-\sigma_{i-1} \ge \kappa$ by construction) we obtain
\begin{align}\label{eq:l2distance}
\begin{split}
&\norm{\prcond{\eta_{\sigma_i}\vert_{S_i}\in \cdot}{T\geq i, A_i}{}- \pi_p^{S_i}}^2_{2,\pi_p^{S_i}} \\
&= \sum_{\xi\in \{0,1\}^{S_i}}\pi_p^{S_i}(\xi) \left(\frac{\prcond{\eta_{\sigma_i}\vert_{S_i}=\xi}{T\geq i, A_i}{}}{\pi_p^{S_i}(\xi)} -1 \right)^2 \\
         &=\sum_{\xi\in \{0,1\}^{S_i}}\pi_p^{S_i}(\xi) \left(\frac{1}{\pi_p^{S_i}(\xi)}\sum_{\eta\in \{0,1\}^{S_i}}\mu_i(\eta) \int_\kappa^\infty \prcond{\til{\eta}_t=\xi}{\til{\eta}_0=\eta}{}\nu_i(dt)     -1
          \right)^2\\
         &=\sum_{\xi\in \{0,1\}^{S_i}}\pi_p^{S_i}(\xi)\left(\int_\kappa^\infty \frac{\prstart{\til{\eta}_t=\xi}{\mu_i}}{\pi_p^{S_i}(\xi)} \nu_i(dt)  -1\right)^2  \\
   &=\sum_{\xi\in \{0,1\}^{S_i}}\pi_p^{S_i}(\xi)\left(\int_\kappa^\infty \left( \frac{\prstart{\til{\eta}_t=\xi}{\mu_i}}{\pi_p^{S_i}(\xi)}  -1 \right)\nu_i(dt)  \right)^2 \\
 &\leq \int_\kappa^\infty \sum_{\xi\in \{0,1\}^{S_i}}\pi_p^{S_i}(\xi)\left( \frac{\prstart{\til{\eta}_t=\xi}{\mu_i}}{\pi_p^{S_i}(\xi)}   -1\right)^2\, \nu_i(dt),
\end{split}     
\end{align}
where we used Jensen's inequality for the last bound.
The spectral gap of the chain $\til{\eta}$ is 1, and hence using Poincar\'e's inequality yields for all $t$
\begin{align*}
        \sum_{\xi\in \{0,1\}^{S_i}}\pi_p^{S_i}(\xi)\left( \frac{\prstart{\til{\eta}_t=\xi}{\mu_i}}{\pi_p^{S_i}(\xi)}   -1\right)^2 &= \norm{\prstart{\til{\eta}_t=\cdot}{\mu_i}-\pi_p^{S_i}}^2_{2,\pi_p^{S_i}} \\
        &\leq e^{-2t} \norm{\mu_i-\pi_p^{S_i}}_{2,\pi_p^{S_i}}^2.
\end{align*}
Plugging this into~\eqref{eq:l2distance} gives
\begin{align}\label{eq:indfirst}
        \norm{\prcond{\eta_{\sigma_i}\vert_{S_i}\in \cdot}{T\geq i, A_i}{}- \pi_p^{S_i}}_{2,\pi_p^{S_i}}\leq e^{-\kappa} \norm{\mu_i-\pi_p^{S_i}}_{2,\pi_p^{S_i}}.
\end{align}
Using~\eqref{eq:muii-1} and the fact that the $L_2$ distance does not increase under projections, we get
\begin{align*}
        \norm{\mu_i-\pi_p^{S_i}}_{2,\pi_p^{S_i}} \leq \norm{\prcond{\eta_{\sigma_{i-1}}\in\cdot}{T\geq i, A_{i-1}}{}    -\pi_p}_{2,\pi_p}.
\end{align*}
Lemma~\ref{lem:23} (with $(\eta_{\sigma_{i-1}},\mathbb{P} \text{ given }T\geq i-1 \text{ and } A_{i-1},\{T\geq i\} \cap A_{i-1})$ here playing the roles of  $(J,\mathbb{P},A)$ from \eqref{e:hatnuL22}, respectively) gives that 
\begin{align}\label{eq:goalinduction}
\begin{split}
        \norm{\prcond{\eta_{\sigma_{i-1}}\in\cdot}{T\geq i, A_{i-1}}{}    -\pi_p}_{2,\pi_p}^2& \\ \leq \frac{1}{(\prcond{T\neq i-1}{T\geq i-1, A_{i-1}}{})^2} &\norm{\prcond{\eta_{\sigma_{i-1}}\in\cdot}{T\geq i-1, A_{i-1}}{}    -\pi_p}_{2,\pi_p}^2 \\&+ \frac{1}{(\prcond{T\neq i-1}{T\geq i-1, A_{i-1}}{})^2} -1. 
        \end{split}
\end{align}
Setting $\theta_i(\cdot) = \prcond{\eta_{\sigma_i}\in \cdot}{T\geq i, A_i}{}$ 
so far we have shown that 
\begin{align}\label{eq:sofar}
        e^{2\kappa}\norm{\theta_i - \pi_p}_{2,\pi_p}^2 \leq & \frac{\norm{\theta_{i-1} - \pi_p}_{2,\pi_p}^2 }{\left(\prcond{T\neq i-1}{T\geq i-1, A_{i-1}}{}\right)^2} \\&+  \frac{1}{(\prcond{T\neq i-1}{T\geq i-1, A_{i-1}}{})^2} -1.
\end{align}
We next show that for all $i$ 
\begin{align}\label{eq:thetai}
        \norm{\theta_i - \pi_p}_{2,\pi_p} \leq \frac{1}{8}.
\end{align}
Since $\eta_0\sim\pi_p$, Lemma~\ref{lem:statindep} gives that conditional on $X_{\sigma_1}$ we have that $\eta_{\sigma_{1}}\sim \pi_p$. Therefore, we get that~\eqref{eq:thetai} is true for $i=1$. Suppose it holds for $i-1$, we show that it also holds for $i$. By the definition of $T$ we have for all $i$
\begin{align*}
&\prcond{T= i}{T\geq i, A_{i}}{} =\prcond{\eta_{\sigma_i}\in \mathrm{\widehat {Env}}(x_i)}{T\geq i, A_{i}}{} \\ &= \sum_{\eta\in \mathrm{\widehat {Env}}(x_i)}  \prcond{\eta_{\sigma_i}\vert_{S_i}=\eta\vert_{S_i}, \eta_{\sigma_i}\vert_{S_i^c}=\eta\vert_{S_i^c}}{T\geq i, A_{i}}{} \\&
= \sum_{\eta\in \mathrm{\widehat {Env}}(x_i)}  \pi_p^{S_i^c}(\eta\vert_{S_i^c})\cdot \prcond{\eta_{\sigma_i}\vert_{S_i}=\eta\vert_{S_i}}{T\geq i, A_{i}}{}, 
\end{align*}
where in the last equality we used the i.i.d.\ property of the different coordinates of $\eta_{\sigma_i}$ and the fact that they are ${\rm{Ber}}(p)$ (properties (1) and (2) from the beginning of the proof). Writing $\lambda_i$ for the projection map from $\{0,1\}^{E}\to \{0,1\}^{S_i}$ and noting that 
\[
\pi_p\left(\mathrm{\widehat {Env}}(x_i)\right)=\sum_{\eta\in \mathrm{\widehat {Env}}(x_i)}  \pi_p^{S_i^c}(\eta\vert_{S_i^c})\cdot \pi_p^{S_i}(\eta\vert_{S_i}) \] 
we get 
\begin{align*}
        &\left|\prcond{T= i}{T\geq i, A_{i}}{} - \pi_p\left(\mathrm{\widehat {Env}}(x_i)\right)\right| \\&\leq \sum_{\eta\in \lambda_i(\mathrm{\widehat {Env}}(x_i))}\pi_p^{S_i}(\eta) \left|\frac{\prcond{\eta_{\sigma_i}\vert_{S_i}=\eta}{T\geq i, A_{i}}{}}{\pi_p^{S_i}(\eta)} -1 \right| \\
        &\leq \norm{\prcond{\eta_{\sigma_i}\vert_{S_i}\in \cdot}{T\geq i, A_i}{} - \pi_p^{S_i}}_{2,\pi_p^{S_i}} \\ & \leq \norm{\prcond{\eta_{\sigma_i}\in \cdot}{T\geq i, A_i}{} - \pi_p}_{2,\pi_p},
\end{align*}
where for the second inequality we used Cauchy Schwartz and for the third one the fact that the~$L_2$ distance does not increase under projections.
From this it now follows that if 
\[
\norm{\prcond{\eta_{\sigma_{i-1}}\in \cdot}{T\geq i-1, A_{i-1}}{} - \pi_p}_{2,\pi_p} \leq \frac{1}{8},
\]
then 
\[
\left| \prcond{T= i-1}{T\geq i-1, A_{i-1}}{} - \pi_p(\mathrm{\widehat {Env}}(x_{i-1}))\right| \leq \frac{1}{8}.
\]
Since $\pi_p(\mathrm{\widehat {Env}}(x_{i-1}))\leq 1/2 $, the above implies that 
\[
\prcond{T= i-1}{T\geq i-1, A_{i-1}}{} \leq \frac{5}{8}.
\]
We are now ready to show that if~\eqref{eq:thetai} holds for $i-1$, then it also holds for $i$. Indeed, 
substituting the above bound into~\eqref{eq:sofar} and using the induction hypothesis $\norm{\theta_{i-1} - \pi_p}_{2,\pi_p} \leq \frac{1}{8} $ give 
\begin{align*}
        \norm{\theta_i-\pi_p}_{2,\pi_p} \leq e^{-\kappa} \left( \frac{(1/8)^2}{(3/8)^{2}}+\frac{1}{(3/8)^2} -1\right)^{1/2} \leq \frac{\sqrt{56}}{3e^{\kappa}},
\end{align*}
which by taking $\kappa$ sufficiently large can be made smaller than $1/8$ and this completes the inductive step and the proof of the lemma. 
\end{proof}

\begin{proof}[\bf Proof of Proposition \ref{prop:LDaux}]
First we claim that for all $\alpha$ we have that 
 $ (1-\alpha)\pi(A(\alpha)^{c}) \le \pi_{\fu,p}(A^c)  $. Indeed, $\pi_{\fu,p}(A) \le  \pi(A(\alpha))+\alpha\pi(A(\alpha)^{c})  $, and hence $\pi_{\fu,p}(A^{c}) \ge (1-\alpha)\pi(A(\alpha)^{c})   $. Since $B=A(1/4)$, this now gives that $ \pi(B^{c})\le \sfrac{4}{3} \pi_{\fu,p}(A^{c}) \le \sfrac{2}{3}$. By increasing $M$ by a $\frac{4}{3}$-factor, we may replace $\pi_{\fu,p}(A^{c})  $ in \eqref{e:wtp2} as well as in the statement of Proposition \ref{prop:LDaux} by $\pi(B^{c})$.
 
 We write $r=r(\delta_0)$ with $\delta_0$ from Lemma~\ref{lem:r} and we let $\alpha=\min((\delta_0/2),1/4)$. 
To simplify notation we write $\Lambda=\Lambda_{{\rm{aux}}, (1,p)}$. 
 Let $\kappa\in \N$ to be chosen later.  
  We now define a sequence of times $(t_j)$ by setting for all $j\geq 0$
 \[
 t_j = j \left(\frac{\log (4\alpha^{-2\kappa}) }{\Lambda(16\alpha^{-2\kappa} /\|\pi_{B^c}- \pi \|_{2,\pi}^2)} \vee r \right)= j\left(\frac{\log (4\alpha^{-2\kappa})  }{\Lambda(16\alpha^{-2\kappa}  \pi(B^c)/\pi(B))} \vee r \right),
 \]
where we write $a\vee b$ for $\max(a,b)$ and $\pi_D$ for $\pi$ conditioned on $D$ (i.e., $\pi_D(x)=\frac{\pi(x)\1{x \in D}}{\pi(D)} $). We will now construct a sequence of random sets $D_0, D_1,\ldots$ such that $B^c\subseteq D_i$ for all $i$ and if $\xi_i = \1(Y_{t_i}\notin D_i)$, then for all $i$ almost surely
\[
\prcond{\xi_i=1}{\xi_0,\ldots, \xi_{i-1}}{}\geq \alpha. 
\]
This will imply the assertion of the proposition by setting 
\[
T_{i}=\inf\left\{ t_j: \sum_{k=0}^{j}\xi_k=i \right\},
\]
i.e.\ $T_i$ is the $i$-th time $t_j$ such that $Y_{t_j}\notin D_j$.

So now we turn to define the sets $D_i$. We do this by induction. For $i=0$ we set $D_0= B^c$. For $i\geq 1$ we will define $D_i$ as a measurable function of $\xi_0,\ldots, \xi_{i-1}$. Since $\pi(B^c)\leq 2/3$ and $Y_0\sim \pi$, we immediately get that 
\[
\pr{T_1=0} =\pr{\xi_0=1}= \pi(B)\geq \frac{1}{3}.
\]
We note that if $Y_0\sim \pi$, then given $Y_0\notin B$, we have that $Y_0\sim \pi_{B^c}$ and similarly if $Y_0\in B$, then $Y_0\sim \pi_B$. We now consider the measures $\nu_i(\cdot ) = \prstart{Y_{t_1}\in \cdot}{\pi_{B^i}}$ for $i=0,1$, where we set $B^0=B$ and $B^1=B^c$. We now argue that
\begin{equation}
\label{e:wenowargue}
\|\nu_0 -\pi\|_{2,\pi}^2 \vee \|\nu_1 -\pi\|_{2,\pi}^2  \le \alpha^{\kappa}\frac{\pi(B)}{\pi(B^c)} .
\end{equation}
Indeed using~\eqref{e:spb3} we obtain 
\begin{equation}
 \label{e:221}
 \|\nu_1-\pi\|_{2,\pi}^2 \;  \le \; \alpha^{\kappa}  \|\pi_{B^c}-\pi\|_{2,\pi}^2 \; = \; \alpha^{\kappa} \frac{\pi(B)}{\pi(B^c)}.
\end{equation}

We now verify that also $\|\nu_0 -\pi\|_{2,\pi}^2 \le  \alpha^{\kappa}\frac{\pi(B)}{\pi(B^c)}  $. Clearly $\|\nu_0-\pi\|_{2,\pi}^2   \le  \|\pi_{B}-\pi\|_{2,\pi}^2 = \frac{\pi(B^c)}{\pi(B)} $. Hence it suffices to consider the case that  $\frac{\pi(B^c)}{\pi(B)} \ge \alpha^{\kappa}\frac{\pi(B)}{\pi(B^c)}  $. In this case, using the fact that $\Lambda(\cdot) $ is non-increasing $\Lambda(16\alpha^{-2\kappa}  \pi(B^c)/\pi(B)) \le \Lambda(4 (4 \alpha^{-\kappa})  /\|\pi_{B}-\pi\|_{2,\pi}^2) $ and so by \eqref{e:spb3} we get that 
\begin{equation}
 \label{e:220}
 \|\nu_0-\pi\|_{2,\pi}^2 \;  \le \; \frac 14 \alpha^{\kappa}  \|\pi_{B}-\pi\|_{2,\pi}^2 \; = \; \frac 14 \alpha^{\kappa}\frac{\pi(B^{c})}{\pi(B)} \le \alpha^{\kappa}\frac{\pi(B)}{\pi(B^c)},
\end{equation} 
as desired, where in the last inequality we  have used the fact that  $\pi(B^c)\leq 2/3$ . 

From~\eqref{e:wenowargue} together with Proposition \ref{prop: Lagrange} we get that $\nu_i(B^c) \le \pi(B^{c})+ \alpha^{\kappa/2} $ for $i=0,1$. Fix $i\in \{0,1\}$. Let $D^i_1 $ be a set in 
\[
\{D \supseteq B^{c}:\nu_i(D) \ge \delta_0/2 \}
\]  
with minimal $\nu_i$ probability. By the definition of $r$ and the fact that $\pi(B^{c}) \le 2/3$ we have that for all $i=0,1$ and for $\kappa$ sufficiently large
\[ 
\nu_i(D^i_1) \in \left[ \frac{\delta_0}{2},\left(\pi(B^{c})+\alpha^{\kappa/2}\right)  \vee\left( \frac{\delta_0}{2}+1-\delta_0\right) \right] \subseteq \left[ \frac{\delta_0}{2}\wedge \frac{1}{4}, \left(1-\frac{\delta_0}{2}\right)\vee \frac{3}{4} \right]. 
\]
To see this, consider the cases $\nu_i(B^c) \le \delta_{0}/2 $ and  $\nu_i(B^c) \ge \delta_{0}/2 $. In the latter we may take $D_1^i=B^c$, while in the former, by minimality $\nu_i(D_1^i) \le \delta_{0}/2 + \max_z \nu_i(z)  $ which is at most $1-\delta_{0}/2$ as $t _1 \ge r $ (using the definition of $r$). We set $D_1=D_1^{0}$ of $Y_0 \in B^0$ and  $D_1=D_1^{1}$ of $Y_0 \in B^1$.
 
This concludes the construction of $D_1$. We now proceed by induction. 
For $a \in \{0,1\}^i $ and $j \le i$  let $a(j) \in \{0,1\}^j $ be the first $j$ co-ordinates of $a$. Assume that for each $a \in \{0,1\}^i$ we have defined sets $D_{j}^{a(j)} \supseteq B^c $ for all $j\leq i$ such that $\nu_a$, the law of $Y_{t_i}$ given that for each $j < i$ we have that $Y_{t_j} \in D_j^{a(j)}$ iff the $j$-th co-ordinate of $a$ is $1$, satisfies that 
\[
\nu_a(D_{i}^{a} ) \in \left[\alpha,1-\alpha\right] \quad \text{ and } \quad \|\nu_a-\pi \|_{2,\pi}^2 \le \alpha^{\kappa}\left( \frac{\pi(B)}{\pi(B^c)}
 + \frac{\alpha^{-2}}{1-\alpha^{\kappa-2}}  \right).
\] 
Note that we have already checked that this holds for our distributions $\nu_i$ for $i=0,1$.
We now want to construct for each $b\in \{0,1\}^{i+1}$ a distribution $\nu_b$ and the set $D_{i+1}^{b}$ such that $\nu_{b}(D^{b}_{i+1}) \in \left[\alpha,1-\alpha\right] $ and $\|\nu_b-\pi \|_{2,\pi}^2 \le \alpha^{\kappa}\left( \frac{\pi(B)}{\pi(B^c)}
 + \frac{\alpha^{-2}}{1-\alpha^{\kappa-2}}  \right) $. For $a\in \{0,1\}^i$, let $ \nu_a^{(0)}$ (respectively, $ \nu_a^{(1)} $) be the measure $\nu_a$ conditioned on $(D_i^a)^{c}$ (respectively, $D_{i}^a $). Then by Lemma~\ref{lem:23} for $j=0,1$ we get
 \begin{align*}
 \| \nu_a^{(j)} - \pi \|_{2,\pi} &\le \frac{ \| \nu_a - \pi \|_{2,\pi}^2 +1}{j\nu_a(D_i^{a})^{2}+(1-j)\nu_a((D_i^a)^{c})^{2}} \le \alpha^{-2}( \| \nu_a - \pi \|_{2,\pi}^2+1) \\
 &\le \alpha^{\kappa-2}   \frac{\pi(B)}{\pi(B^c)}+\frac{\alpha^{-2}}{1-\alpha^{\kappa-2}}=:M.  
 \end{align*}
 Using $\pi(B^c)\leq 2/3$ we see that $M \le    \frac{\pi(B)}{\pi(B^c)} (\alpha^{\kappa-2}+2\frac{\alpha^{-2}}{1-\alpha^{\kappa-2}}) \le 3 \frac{\pi(B)}{\pi(B^c)}\alpha^{-2} $, provided $\kappa \ge \kappa_0(\alpha) $. Using this bound, we see that  
 \[
 t_1\geq \frac{\log (\alpha^{-2\kappa+2})}{\Lambda\left( 4\alpha^{-2\kappa+2}/M \right)}.
 \]
For $b\in \{0,1\}^{i+1}$ we now define $\nu_b(\cdot) = \prstart{Y_{t_{1}}\in\cdot}{\nu_{b(i)}^{(b_{i+1})}}$ with $b_{i+1}$ denoting the $i+1$-st coordinate of~$b$ (and $b(i)$ its  first $i$ co-ordinates). Using~\eqref{e:spb4} this time we obtain that for sufficiently large $\kappa$
 \[
 \| \nu_b-\pi\|_{2,\pi}^2 \; \le \; \alpha^{2\kappa-2}
M \; \le \; \alpha^{\kappa}\left( \frac{\pi(B)}{\pi(B^c)}
 + \frac{\alpha^{-2}}{1-\alpha^{\kappa-2}}  \right) .  
 \] 
In particular $\| \nu_b-\pi\|_{2,\pi}^2 \le 3 \alpha^{\kappa-2} \frac{\pi(B)}{\pi(B^c)}  $ provided $\kappa$ is sufficiently large. 
Proposition~\ref{prop: Lagrange} gives that $ \nu_b(B^c) \le \pi(B^{c})+ \sqrt{3} \alpha^{(\kappa-2)/2} \le \pi(B^{c})+\alpha^{\kappa/4}  $ for $\kappa$ sufficiently large. In the same way as above when defining the set~$D_1^i$ we get that if $\kappa$ is sufficiently large, there exists some set $D_{i+1}^b \supseteq B^{c} $ such that $ \nu_b(D_{i+1}^b) \in [\alpha,1-\alpha] $. This completes the induction and the proof of the proposition.
\end{proof}

\subsection{Adaptations for the dynamical random rates model}
\label{s:randomrates}
 We now sketch the necessary adaptations to the argument required to analyze the random walk on dynamical random rates model and derive the results mentioned in Remark (9) on page 7. Recall the definition of the model as well as the claimed results.

Here each environment $\eta$ specifies the rates of the edges, not merely whether the rate of an edge is $0$ or positive. Hence the state space of the full process is potentionally much larger than in the setup of Thereom \ref{thm:1}. 
As we employ the spectral-profile technique, as well as \eqref{e:laexit}, we wish to reduce the problem to the case that $\nu$ has finite support, and thus the state space is finite. (We note that the spectral-profile technique extends to the case that the state space is continuous, but in that case it requires some mild regularity. Instead of verifying the relevant regularity condition, as well as verifying the validity of \eqref{e:laexit} in this setup, we simply reduce the problem to the case $\nu$ has finite support.)

We approximate the measure $\nu$ by a sequence of measures $\nu_n$ of finite support, also satisfying the assumptions on $\nu$. We now argue that a uniform bound on the total variation mixing times of the walk co-ordinate corresponding to the $\nu_n$'s (we will actually bound the mixing time for the corresponding full process) implies the same bound on the total variation mixing time of the walk co-ordinate for $\nu$. As we now explain, this follows from a straightforward coupling argument:
\\ \noindent Let $T,\delta >0$. Consider two sequences of time evolving environments $(\eta_{t})_{t \ge 0}$ and $(\eta'_{t})_{t \ge 0}$ such that for all $t \le T$ and all $e$ we have that $|\eta_t(e)-\eta'_t(e)| \le \delta$. Then starting from the same initial state, the walks on these two evolving environments can be coupled so that they are equal to one another by time $T$ with probability at least $1-e^{-T\delta d}$ (we omit the details).

 In our application, we can take $T$ to be the upper bound on the total variation mixing time corresponding to the $\nu_n$'s (which is independent of $n$), and so by taking $n$ large enough, we can pick $\delta$ to be arbitrarily small, and have that indeed the environments corresponding to $\nu$ and to $\nu_n$ could be coupled as above (as to only differ by at most $\delta$) by time $T$ with probability arbitrarily close to $1$.

Denote the support of $\nu$ by $S$ and assume it is finite. Crucially, $\pi$ is a stationary distribution for the walk for all possible environments, and the stationary distribution of the full process is given by $\pi \times \nu^{\otimes E}$.

As in the proof of Theorem \ref{thm:1} we wish to define regeneration times, and through them an auxiliary chain, defined as the position of the walk co-ordinate at the regeneration times. We wish the law of the environment at a regeneration time to be stationary, i.e.\ $\nu^{\otimes E}$, independently of the trajectory of the walk co-ordinate until that time.
 One difficulty is that in the random rates model, when the walk co-ordinate is at $v$, some information is gathered on the rates of \textbf{all} the edges incident to $v$, and it is gathered even before the walk jumps. Thus if at most one edge is refreshed at each update, there is always some information on the environment at times at which the walk is not at a degree $1$ vertex of $G$.

This motivates the following variant of the model: for all $v \in V$ (independently) at rate $\mu$ all of the edges incident to $v$ are refreshed according to the law $\nu$ (independently).
Let $\mathcal{E}$ and $\mathcal{E}'$ be the Dirichlet forms corresponding to the original random walk on dynamical random rates model (with single edge updates) and to the aforementioned variant, respectively. We argue that for all $f:V \times S^{E} \mapsto \mathbb{R}$ we have that
\begin{equation}
\label{e:singletoblock}
2d \mathcal{E}(f,f) \ge \mathcal{E}'(f,f).
\end{equation}
By \eqref{e:singletoblock}, at the price of picking up a $2d$ factor, it suffices to consider the aforementioned variant. Indeed the analysis below relies on the spectral-profile, which is amenable to such comparisons. 

We now prove \eqref{e:singletoblock}. Since transitions along the walk co-ordinate give exactly the same contribution in $\mathcal{E}(f,f)$ and in $\mathcal{E}'(f,f)$, it is enough to prove this for the Dirichlet forms corresponding to the evolution of the environments, which by abuse of notation we also denote by $\mathcal{E}$ and $\mathcal{E}'$. To compare these Dirichlet form, we write them in a convenient form. We first need some notation.

Let $e_1,\ldots,e_k \in E$.
Let $\eta \sim \nu^{\otimes E}$. Let  $\eta^{e_1} \sim \nu^{\otimes E}$ be such that $\eta(e)=\eta^{e_1}(e)$ for all $e \neq e_1$ and $\eta(e_1)$ and $\eta^{e_1}(e_1)$ are independent. Likewise, let  $\eta^{e_1e_2} \sim \nu^{\otimes E}$ be such that $\eta^{e_1}(e)=\eta^{e_1e_2}(e)$ for all $e \neq e_2$ and $\eta^{e_1}(e_2)$ and $\eta^{e_1e_2}(e_2)$ are independent. We define $\eta^{e_1\cdots e_i}$ analogously (in terms of $\eta^{e_1\cdots e_{i-1}}$) by induction. For each $v \in V$ let $e(v_1),\ldots,e(v_d)$ be the edges incident to $v$.
By the Cauchy-Schwarz inequality, for all $f:S^{E} \mapsto \mathbb{R}$ we have that
\[\mathcal{E}'(f,f)=\mu \sum_{v \in V} \mathbb{E}\left[\left(f(\eta)-f(\eta^{e(v_1) \cdots e(v_d)})\right)^2 \right] \]
 \[ \le \mu \sum_{v \in V}d\sum_{i \in [d]} \mathbb{E}\left[\left(f(\eta^{e(v_1)\cdots e(v_{i-1})})-f(\eta^{e(v_1) \cdots e(v_i)})\right)^2 \right] \]
\[=2d \mu \sum_{e \in E} \mathbb{E}\left[\left(f(\eta)-f(\eta^{e})\right)^2 \right]=2d\mathcal{E}(f,f) \]

One adaptation to the proof of Theorem \ref{thm:1} is that instead of considering $\mu=1$ and comparing smaller $\mu$'s to it, here we work with $\mu=M:=C'(a,b)\frac{d\mathbb{E}[X^2 \mid X>0]}{\mathbb{E}[X]}$, for some constant $C'(a,b)$ depending only on $(a,b)$ (to be chosen later), and compare smaller $\mu$'s to it (larger values of $\mu$ can be analyzed in the same fashion the case $\mu=M$ is treated).
(Recall that $a$ and $b$ are the constants from our assumption that $\mathbb{E}[e^{a(X/\mathbb{E}[X]})] \le b$, where $X \sim \nu$.)

 Another difference is in the construction of the regeneration times (and thus of the auxiliary chain). Instead of having infected edges (which are the ones on which we currently have some information about their rate) as we did in the proof of Theorem  \ref{thm:1}, now we have infected vertices. We declare $v$ infected when the walk jumps to $v$, or if the walk is currently at $v$ (if the walk jumps to an infected vertex, the number of infected vertices does not increase). When $v$ is picked to have the edges incident to it refreshed, we remove $v$ from the set of infected vertices. However, if the walk is at $v$ at that moment, it becomes infected again right after that time. 

The $(i+1)$th regeneration time denoted by $\tau_{i+1}$ is defined inductively to be the first time $t$ after time $\tau_i$ (with the convention that $\tau_0=0$ and that at time $0$ all vertices are infected) such that
\begin{itemize}
\item For some $s \in (\tau_i,t)$ we have that the walk co-ordinate at time $s$ is at a different location than its location at time $\tau_{i}$;
\item The only infected vertex at time $t-$ is the location of the walk at time $t$, denoted by $v$, and at time $t$ the edges indiced to $v$ are refreshed. 
\end{itemize}
(Note that right after a regeneration time, the vertex at which the walk is currently at becomes infected again.)

By construction, it is indeed the case that at a regeneration time the law of the environment is its stationary distribution $\nu^{\otimes E}$, independently of the trajectory of the walk until that time, and that the spacings between regeneration times are i.i.d.\ (apart from the first regeneration time, which has a different law).

As stated above, we fix $\mu=M$. Recall that $\kappa=\kappa(\mu)$ is the expected time required for the walk to leave its location after a regeneration time. By the assumption that $\mathbb{E}[e^{a(X/\mathbb{E}[X])}] \le b$, the constant $C'(a,b)$ in the definition of $M$ can be picked as to ensure that $\frac{\tau_{2}-\tau_1}{\kappa}$ has mean $O(1)$ and an exponentially decaying tail, and the rate of exponential decay can be bounded independently of $G$ and of $\nu$. (The last fact is essential in ensuring that the constant $C(a,b)$ in our bound on the mixing time depends only on $(a,b)$, but not on $G$ and $d$). We omit the details.

The auxiliary chain $(Y_i)_{i=0}^{\infty}$ is now defined by setting $Y_i$ to be the location of the walk at the $i$th regeneration time. By symmetry, its stationary distribution is uniform (this is the only use of transitivity). Crucially, by the definition of $M$ it is not hard to verify that the following analog of Lemma \ref{lem:PauxPsrw} holds also in the current setup (for $\mu \ge M)$:
\[ \frac{P_{\mathrm{aux},(\mu,p)}(x,y)+P^*_{\mathrm{aux},(\mu,p)}(x,y)}{2} \gtrsim 1/d. \]
This follows by considering the case that the previous regeneration time occurred when the walk was at vertex $x$, then the walk co-ordinate moved to $y \sim x$, and then before the walk co-ordinate moved away from $y$ the following two things occurred: the edges incident to $x$ were refreshed, and then also the edges incident to $y$ were refreshed.
We omit the details of this calculation.


We conclude with a few technical remarks:

For $A \subset V \times S^{E}$, $a \in V$ and $\alpha \in [0,1]$ we can define
$\mathrm{Env}(a,A) \subseteq \mathbb{R}_+^{E}$ and $A(\alpha) \subseteq V$ as in Definition \ref{def:EF}, by replacing $\pi_p$ with $\nu^{\otimes E}$ and $\{0,1\}^E$ by $S^E$. Crucially, because the stationary distribution of the full process is again given by a product measure, we can again bound $\pi (A(\alpha))$ from below, exactly as we have in the proof of Theorem \ref{thm:1}.

A key observation is that here the $p$-tilted hypercube $\{0,1\}^E$ is replaced by a process on $S^E$ corresponding to the evolution of the environment (in the modified model, in which at rate $\mu$ all of the edges incident to a vertex are refreshed, independently). When bounding the $L_2$ mixing time of this chain (and then multiplying by $d$, as required when translating the result to the original model), we pick up the additive term $\frac{Cd|\log \min_{x:x \neq 0}\nu(x) |}{ \mu}$ from \eqref{e:rrmt}. (To be precise, we pick up a $\frac{Cd\log \left(|E|/ \min_{x}\nu(x) \right)}{ \mu}$ term, but by absorbing part of it into the first term in the r.h.s.\ of \eqref{e:rrmt}, we only get an $\frac{Cd|\log \min_{x:x \neq 0}\nu(x) |}{ \mu}$ term.) 
To obtain an upper bound on the total variation mixing time which does not depend on $\min_{x \neq 0}\nu(x)$ we use Remark \ref{r:TV}. 

Crucially, like the $p$-tilted hypercube of  rate $\mu$, the last Markov chain has spectral gap at least $\mu$. This fact is crucial in extending Lemma \ref{lem:keyhyper} to the current setup. The analog of this lemma in our current setup is then used to relate the rate of exponential decay of the law of the hitting time of a set $A \subset V \times S^E$ (by the full process) to that of the set $A(1/4) \subset V$ (by the auxiliary chain). This is then used to bound the spectral profile of the full process in terms of that of SRW on $G$. The details are analogous to those in the proof of Theorem \ref{thm:1}.  

\section{Applications: transitive graphs of moderate growth and the hypercube} 
\label{s:lower}

In this section we prove Theorems~\ref{thm:mod} and~\ref{thm:hyp}. Below we identify a percolation cluster with the vertices lying in it. Recall that we  denote the cluster of vertex $x$ by $K_x$ and its edge boundary by~$\partial K_x $.
Finally, recall that we write $M_p= \pi_p(|\partial K_x||K_{x}|^{2})   $ and $N_p=\pi_p(|K_{x}|)$.

\begin{lemma}
\label{lem:mod}
Let $c,a\in \R$. There exists a positive constant $c_1=c_1(a,c)$ such that the following holds. Let $G=(V,E)$ be a connected vertex-transitive graph of $(c,a)$-moderate growth and diameter~$\gamma$. Suppose that $N_p\leq \gamma/4$. Then we have 
\[
\rel^{\mathrm{full},(\mu,p)} \geq c_1 \frac{( \gamma-4N_p)^2}{\mu pM_{p}}.
\]
\end{lemma} 

\begin{proof}[\bf Proof]

Denote the cluster of $x \in V$ w.r.t.\ $\eta$ by $K_x(\eta)$ and as usual we identify it with the set of vertices lying in it. 

Fix some $o \in V$. 
Let $f(x,\eta)=\sfrac{1}{|K_{x}(\eta)| }\sum_{v \in K_{x}(\eta) } d_G(v,o) $, where $d_G$ is the graph distance w.r.t.\ $G$. For $\eta \in \{0,1\}^E $ denote by $\eta^e $ the environment obtained from $\eta$ by setting $e$ to be open if it was not already open (i.e.\ $\eta^e(e)=1$ and $\eta^e(e')=\eta(e')$ for all $e' \in E \setminus \{ e\}$). Observe that the value of $f$ cannot change as a result of a jump in the random walk co-ordinate  (as such a jump leaves the walk in the same percolation cluster). Thus
\begin{equation}
\label{e:testDF}
\begin{split}
\EE(f,f) & =\sum_{x \in V ,\, \eta \in \{0,1\}^E , \, e \in \partial K_x(\eta) }\pi(x) \pi_{p}(\eta) \mu p\left(f(x,\eta)-f(x,\eta^e) \right)^2
\\ & \le \mu p \sum_{x,\eta,e \in \partial K_x(\eta) }\pi(x) \pi_{p}(\eta)|K_{x}(\eta^{e})|^2 \\ & = \mu p \sum_{\eta,e \in \partial K_o(\eta) } \pi_{p}(\eta)|K_{o}(\eta^{e})|^2 \\ & =\mu p \sum_{x,y : \, \{x,y\} \in E }\pi_{p}\left[\1(y \notin K_o \ni x ) (|K_o |+|K_y |)^2 \right]   ,
\end{split}
\end{equation}
where the penultimate equality follows from transitivity, and the last one from the fact that on $\{y \notin K_o \ni x\}$ we have that $|K_o^{\{x,y\}} |=|K_o |+|K_y | $. 

Observe that for all $A \subset V $ with $o,x \in A $ and $y \notin A$, given that $K_o=A $ we have that $|K_y|$ is distributed as the size of the percolation cluster of $y$ (with parameter $p$) on the induced graph on $V \setminus A$, which by an obvious coupling with $\pi_p$ is stochastically dominated by the (unconditional) law of $|K_y|$ under $\pi_p$. Thus for all such $A$ we have that $\mathbb{E}_{\pi_p}[|K_y |^a \mid   K_{o}=A] \le \mathbb{E}_{\pi_p}[|K_y |^a]  $, and hence for all $a>0$ we have that
\begin{eqnarray}
\label{e:BK}
\mathbb{E}_{\pi_p}[\1( y \notin  K_o \ni x)|K_y |^{a} \mid   K_{o}] & \le & \mathbb{E}_{\pi_p}[|K_y |^{a}] \; \; = \; \; \mathbb{E}_{\pi_p}[|K_o |^{a}].
\end{eqnarray}
Plugging  \eqref{e:BK} in \eqref{e:testDF} and summing over $\{x,y\} \in E $ yield 
\begin{eqnarray}
\label{e:BKK}
\EE(f,f) & \le & 4\mu pM_{p}. 
\end{eqnarray}

We conclude the proof by showing that $\Var_{\pi_{\fu,p}}(f) \ge c_{a,c} (\gamma-4N_p)^{2} $. It is not hard to verify that 
\begin{equation}
\label{e:dGf}
\forall \; x, \qquad |d_{G}(x,o) - \sum_{\eta}\pi_p(\eta) f(x,\eta) | \; \le \; \pi_p(|K_x|) \; = \; N_p. \end{equation}  
Let $(X,\eta)\sim \pi\times \pi_p$ and $(Y,\eta') \sim \pi\times \pi_p$ be independent. Define 
\[
A=\{d_G(Y,o) \le  \gamma/4, \quad  d_G(X,o) \ge 3\gamma/4 \}.
\]
Since $G$ is of moderate growth, there exists a positive constant $b$ such that $\pr{A} \ge b$.  Finally, by the independence between $X$ and $\eta$ together with \eqref{e:dGf}  $\E[ f(X,\eta) \mid A ] \ge 3\gamma/4-N_{p}$ and similarly we also have that  $\E[ f(Y,\eta) \mid A ] \le \gamma/4+N_{p}$, which together yield   that
\begin{align*}
        2\Var_{\pi_{\fu,p}}(f)&=\E[(f(X,\eta)-f(Y,\eta'))^2 ]\ge \E[(f(X,\eta)  -f(Y,\eta'))^2 \1(A)]   \\ 
&\ge (\pr{A})^2  \left(\E[ f(X,\eta) \mid A ]  -\E[f(Y,\eta')\mid A ])\right)^2  \ge b^2(\gamma-4N_p)^2/4,   
\end{align*}
 where for the second inequality we used Jensen's inequality and for the last one we used the assumption that $\gamma\geq 4N_p$.
\end{proof}

\begin{proof}[\bf Proof of Theorem~\ref{thm:mod}]

        Let $P$ be the transition matrix of simple random walk (SRW) on $G$. Diaconis and Saloff-Coste \cite{moderate} showed that for a  Cayley graph $G$ of $(c,a)$-moderate growth we have  
        \[
c^2\gamma^2 4^{-2a-1} \le \rel \lesssim t_{\mathrm{mix}}^{(\infty)} \lesssim_{c,a}\gamma^2.   \]  
Using this, Lemma~\ref{lem:mod} and the assumptions on $M_p$ and $N_p$ we have 
\[
\rel^{\mathrm{full},(\mu,p)}\gtrsim_{a,b,c} \frac{1}{\mu p} \gamma^2 \asymp_{a,c}\frac{1}{\mu p} \rel^{\SRW}.
\]
As proved in \cite[Proposition 8.1]{exclusion} and previously noted in \cite{Lyonsev}, for vertex-transitive graphs of degree~$d$ and $(c,a)$-moderate growth one has that \newline $ t_{\mathrm{spectral-profile}}^{\SRW}\lesssim_{a,c,d}\gamma^2 $. 
This together with Theorem~\ref{thm:1} yields
\begin{align*}
\rel^{\mathrm{full},(\mu,p)}\lesssim \mixfmupi \lesssim_{a,c,d} \frac{1}{\mu p} \gamma^2 + \frac{1}{\mu}|\log \left(1-p \right)|      \lesssim \frac{1}{\mu p} \gamma^2 \asymp \frac{1}{\mu p} \rel^{\SRW},
\end{align*}
where for the last inequality we used the assumption $|\log \left(1-p \right)|   \leq \gamma^2$. This completes the proof.
\end{proof}

\begin{proof}[\bf Proof of Theorem~\ref{thm:hyp}]
        The proof follows from Theorem~\ref{thm:1} together with the upper bound on the spectral profile of simple random walk on the hypercube (see \cite[Section~7]{exclusion}).
\end{proof}

\section{Log-Sobolev constant}
 Theorem 1.2 in \cite{HP} asserts that for every Markov chain on a finite state space \begin{equation}
  \label{HPLSchar}
\frac{1}{17} \le \inf_{\eps \in (0,1/2]} \frac{\log(1/\eps)c_{\mathrm{LS}}}{ \Lambda_{0}(\eps)
} \le 1. 
 \end{equation}
We note that the result in \cite{HP} is stated for the case that $\cL=P-I$ for some transition matrix~$P$, but as noted several times before, the general case can be reduced to the case  $\cL=c(P-I)$ for some $c>0$, and the relevant quantities scale linearly in $c$. We also note that the equivalence between \eqref{HPLSchar} and Theorem 1.2 in \cite{HP} relies on the general fact that if $A=\cup_{i=1}^r A_i $ and the sets $A_i$  are disjoint and satisfy $\cL(x,y)=0$ for all $x \in A_i$ and $y \in A_j$ for $i \neq j$, then (e.g.\ by \eqref{e:laAextremalchar})
\begin{equation}
  \label{llambdaunion}    
\lambda(A)=\min_{i \in [r] }\lambda(A_i).
\end{equation}
As for every singleton $x $ by \eqref{e:laAextremalchar} we have that $\lambda(\{x\})=-\cL(x,x) $ it follows that
\begin{equation}
  \label{LLStrivialbound}
 c_{\mathrm{LS}}\le \min_x \frac{- \cL(x,x)}{ \log(1/ \pi(x))}.
 \end{equation}

 \begin{proof}[\bf Proof of Theorem \ref{thm:cLS}]
Let $M  $ be as in \eqref{e:wtp}. Recall that   $\Lambda_{0}(\eps) $ and   $\Lambda(\eps) $ are non-increasing in $\eps$. Thus   by increasing $M$ if necessary, we may assume that $M \ge 2 $. Using this again, as well as   $  \Lambda_{0}(1/2) \le \Lambda(1) \le 2 \Lambda_{0}(1/2)$  (this is used to treat $\eps \in (\pi_*/M,\pi_*] \cup (\frac{1}{2M},\frac 12]$; for the first inequality see the proof of \cite[Lemma~2.2]{spectral},  the second inequality follows from \eqref{e:lala0} by monotonicity),  by \eqref{e:wtp} together with \eqref{HPLSchar},  we have that
\[\mu p \min_{\eps \in [\pi_*/M, 1/2] } \frac{\log(1/\eps)}{ \Lambda_{0}^{\fu,(\mu,p)}(\eps)
} \lesssim  \min_{\eps \in [\pi_*, 1/2] }  \frac{\log(1/\eps)}{ \Lambda_{0}^{\SRW}(\eps)
} \asymp 1/c_{\mathrm{LS}}^{\SRW}.  \]

It is left to prove $\min_{\eps \in (0,\pi_*/M] } \frac{\log(1/\eps)}{ \Lambda_{0}^{\fu,(\mu,p)}(\eps)} \lesssim \frac{\log (1/\pi_*)\log(\frac{1}{p(1-p)}) }{\mu }  $.

Let $\delta \le 1/M$ and let $B$ be such that  $\pi_{\fu,p}(B) =\delta \pi_*  $. Then for all $v \in V $ we have that
\[
\pi_p(\mathrm{Env}(v,B)) \le \frac{\sum_{u}\pi(u)\pi_p(\mathrm{Env}(u,B))}{\pi_*}=\frac{\pi_{\fu,p}(B) }{\pi_*} = \delta. 
 \]

Consider now the (reducible) Markov chain on $\Omega:=V \times \{0,1\}^E$ in which the walk co-ordinate cannot change, and the environment evolves in the usual fashion by refreshing each edge at rate $\mu$ (and declaring it to be open with probability $p$ and closed w.p.\ $1-p$). Note that it is reversible w.r.t.\ $\pi_{\fu,p}$. Denote the corresponding generator by $\cL'$ while that of the (usual) full process with parameters $(\mu,p)$ by $\cL$. Let $\lambda'(B)$ and $\lambda(B)$ be the  minimal Dirichlet eigenvalue of $B$ w.r.t.\ $\cL'$ (i.e., the minimal positive eigenvalue of $\cL_B'$, given by $\cL'_B(a,b):=\cL'(a,b)\1(a,b \in B)$) and $\cL$, respectively.  Then by \eqref{e:laAextremalchar} we have that
\begin{equation}
\label{e:la(B)la'(B)}
\begin{split}
\lambda'(B)&=\min \{ \EE_{\cL'}(h,h): h \in \R_{+}^{\Omega},\, \|h\|_2=1,\, \mathrm{supp}(h) \subseteq B    \} \\ 
& \le  \min \{ \EE_{\cL}(h,h): h \in \R_{+}^{\Omega},\, \|h\|_2=1,\, \mathrm{supp}(h) \subseteq B    \}=\lambda(B).
\end{split}
 \end{equation}  
(The first equality holds even without irreducibility).
Since w.r.t.\ $\cL'$ the sets $(\{v\} \times \mathrm{Env}(v,B):v \in V) $ are disconnected (i.e.\ $\cL'(a,b)=0 $ for all $a \in \{v\} \times \mathrm{Env}(v,B) $ and  $b \in \{u\} \times \mathrm{Env}(v,B) $, for all $v \neq u$) and their union is $B$, by \eqref{llambdaunion}   we have that 
\begin{equation}
\label{e:la'(B)minv}
\lambda'(B)= \min_v \lambda'(\{v\} \times\mathrm{Env}(v,B)).
 \end{equation}
As the evolution of the chain corresponding to $\cL'$ on each set of the form $\{v\} \times \{0,1\}^E $ is simply that of the rate $\mu$ $p$-tilted hypercube, for every $v$ we have that
\begin{equation}
\label{e:la'fiberhyper}
\lambda'(\{v\} \times\mathrm{Env}(v,B))=\lambda_{p\text{-tilted hypercube with rate }\mu}(\mathrm{Env}(v,B)).
 \end{equation}
 As $\pi_p(\mathrm{Env}(v,B)) \le \delta   $, by \eqref{HPLSchar}  we have that for all $v$ (uniformly in $p$ and $\mu$) 
\begin{equation}
\label{e:hyperLS}
\lambda_{p\text{-tilted hypercube with rate }\mu}(\mathrm{Env}(v,B)) \gtrsim \hat c \log(1/\delta) ,
\end{equation}
 where $\hat c = \hat c(\mu,p) $ is the log-Sobolev constant of the rate $\mu$ $p$-tilted hypercube (this is a dimension free quantity \cite{diaconis}, but here the co-ordinates of the hypercube are labeled by the set $E$). Finally, combining  \eqref{e:la(B)la'(B)}-\eqref{e:hyperLS} and using the facts that   $\pi_{\fu,p}(B) =\delta \pi_*  $ and  $\delta \le 1/M \le 1/2 $ we see that
\begin{align*}
\frac{\lambda(B)}{ \log(1/\pi_{\fu,p}(B))} \gtrsim \frac{\hat c \log(1/\delta)}{ \log(1/(\delta \pi_*) )  } &\gtrsim \hat c/\log(1/ \pi_* )\\&\gtrsim \frac{ \mu (1-2\min(p,1-p)) }{\log(\max(p,1-p)/\min(p,1-p)) )\log(\frac{1}{\pi_*} ) }, 
\end{align*}
where in the last inequality we have used 
\[
\hat c \asymp \mu(1-2\min(p,1-p))/\log(\frac{\max(p,1-p)}{\min(p,1-p))} 
\]
 (e.g.\ \cite{diaconis}, alternatively, this can be seen using the facts that (i) the log-Sobolev constant scales linearly in $\mu$, (ii) for product chains it is the same as the minimal  log-Sobolev constant of a single co-ordinate \cite{diaconis}, and (iii) the log-Sobolev constant of a single co-ordinate can be approximated using \eqref{HPLSchar}). We note that for $p=1/2$ this should be interpreted as $(1-2\min(p,1-p))/\log(\frac{\max(p,1-p)}{\min(p,1-p))}=1 $.

We now prove the result about $\rel$. Combining \eqref{e:lambdaA123} and \eqref{e:wtp} we have that  \[(\mu p)^{-1} \Lambda_{(\mu,p)}(1/2) \gtrsim \min_{a} \Lambda_{\SRW}(a) = (\rel^{\SRW})^{-1}.
\] 
Finally, using the fact that for every reversible chain $\Lambda(1/2) \leq {2}/{\rel}  $ \cite[Lemma~2.2]{spectral} completes the proof.  
\end{proof}

\noindent \textbf{Acknowledgments:} We are grateful to Tom Hutchcroft for several useful suggestions.

\bibliographystyle{abbrv}
\bibliography{DP}

\end{document}